\documentclass[12pt]{article}

\usepackage{tikz}
\usepackage{subfigure}
\usepackage[english]{babel}
\usepackage{mathrsfs}
\usepackage[center]{caption2}
\usepackage{amsfonts,amssymb,amsmath,latexsym}
\usepackage{multirow}
\usepackage{amsthm}

\def\g_p{\gamma_{pr}}

\newtheorem{thm}{Theorem}[section]

\newtheorem{lem}[thm]{Lemma}
\newtheorem{prop}[thm]{Proposition}
\newtheorem{conj}[thm]{Conjecture}

\newtheorem{claim}{Claim}
\newtheorem{fact}{Fact}

\newtheorem{cons}[thm]{Construction}

\begin{document}
\title{Codegree conditions for tilling balanced complete $3$-partite $3$-graphs and generalized 4-cycles
%\thanks{The work was supported by NNSF of China (No. 11671376) and  NSF of Anhui Province (No. 1708085MA18).}
}
\author{Xinmin Hou$^a$,\quad Boyuan Liu$^b$, \quad Yue Ma$^c$, \\
\small $^{a,b,c}$ Key Laboratory of Wu Wen-Tsun Mathematics\\
%\small Chinese Academy of Sciences\\
\small School of Mathematical Sciences\\
\small University of Science and Technology of China\\
\small Hefei, Anhui 230026, China.\\
}

\date{}

\maketitle
\begin{abstract}
Given two $k$-graphs
$F$ and $H$, a perfect $F$-tiling
(also called an $F$-factor)
in $H$ is a set of vertex disjoint copies of $F$ that together cover the vertex set of $H$. Let $t_{k-1}(n, F)$ be the smallest integer $t$ such that every  $k$-graph $H$ on $n$ vertices with minimum codegree at least $t$ contains a perfect $F$-tiling.  Mycroft (JCTA, 2016) determined  the asymptotic values of $t_{k-1}(n, F)$ for $k$-partite $k$-graphs $F$.
Mycroft also conjectured that the error terms $o(n)$ in $t_{k-1}(n, F)$ can be replaced by a constant that depends only on $F$.
In this paper, we improve the error term of Mycroft's result
to a sub-linear term when $F=K^3(m)$, the complete $3$-partite $3$-graph with each part of size $m$. We also show that the sub-linear term
is tight for $K^3(2)$, {the result also provides another family of counterexamples of Mycroft's conjecture (Gao, Han, Zhao (arXiv, 2016) gave a family of counterexamples when $H$ is a $k$-partite $k$-graph with some restrictions.)} Finally, we show that  Mycroft's conjecture holds for generalized 4-cycle $C_4^3$ (the 3-graph on six vertices and four distinct edges $A, B, C, D$ with $A\cup B= C\cup D$ and $A\cap B=C\cap D=\emptyset$), i.e. we determine the exact value of $t_2(n, C_4^3)$.
\end{abstract}

\section{Introduction}
A {\it $k$-graph} $H$ is a pair $(V, E)$, where $V$ is a set of vertices and $E$ is a collection of subsets of $V$ with uniform size $k$, we call $|V|$ {\it the order} of $H$ and $|E|$ {\it the size} of $H$, also denoted by $|H|$.  We write graph for $2$-graph for short.
Given two $k$-graphs $F$ and $H$, an {\em $F$-tiling} in $H$ is a collection of vertex disjoint copies of $F$ in $H$.
%such that each copy is isomorphic to $F$.
An $F$-tiling is {\it perfect} if every vertex of $H$ is covered.
%the number of copies of $F$ is $V(H)/V(F)$.
A perfect $F$-tiling in $H$ is also called an {\it $F$-factor}.
If $F$ is a single edge then an $F$-factor in $H$ is a perfect matching in $H$. As for matchings, a natural question for tilling is to determine the minimum degree threshold for finding a perfect $F$-tiling.
Given a subset $S$ of vertices in a $k$-graph $H$, the degree of $S$, denote by $d_H(S)$, is the number of edges of $H$ containing $S$. The minimum $s$-degree $\delta_s(H)$ of $H$ is the minimum of $d_H(S)$ over all $S\subseteq V(H)$ of size $s$.
%For $k$-graphs $F$ and $H$, here we investigate the minimum $\delta_s(H)$ thresholds that force a perfect tilling of $F$.
For integer $n$ divisible by $|V(F)|$,
%a $k$-graph $F$ and an integer n divisible by $|V(F)|$,
define {\it $t_s(n, F)$} to be the smallest integer $t$ such that every  $k$-graph $H $ on $n$ vertices with $\delta_{s}(H)\geq t$ contains a perfect $F$-tiling.
%{\color{red}We always assume that $|V(F)|$ divides $n$ since $t_s(n,F)=\binom{n-s}{k-s}$ when $n$ is not divisible by $|V(F)|$. ???0402}
We write $[n]$ for the set $\{1,\ldots,n\}$, $n\in\mathbb{N}$ and $r\mathbb{N}$ for the set of positive integers divisible by integer $r$

Tiling problems have been widely studied for graphs. The celebrated Hajnal-Szemer\'edi  theorem~\cite{Hajnal-Szemeredi} states that $t_1(n, K_r)=(1-1/r)n$ for $n\in r\mathbb{N}$. Alon and Yuster~\cite{Alon-Yuster} generalized Hajnal-Szemer\'edi  theorem to $t_1(n, H)=(1-1/\chi(H)+o(1))n$ for every $H$ with chromatic number $\chi(H)$ and order $h$ and $n\in h\mathbb{N}$; later, Koml\'os, S\'ark\"ozy, and Szemer\'edi~\cite{K-S-S} proved that the error term $o(1)$ can be replaced by a constant $c=c(H)$.
%every graph $G$ on $n\in r\mathbb{N}$ vertices and with $\delta(G)\geq(1-1/r)n$ contains a $K_r$-factor.
In~\cite{Kuhn-Osthus}, K\"uhn and Osthus  improved Alon-Yuster result to $t_1(n, H)\leq(1-1/\chi^*(H))n+O(1)$, where $\chi^*(H)$ depends on the relative sizes of the colour classes in the optimal colourings of $H$ and satisfies $\chi(H)-1\le\chi^*(H)\le \chi(H)$. See \cite{survy for graph tilling} for a survey on graph tiling.

For $k\geq3$, we know much less and tiling problem becomes much harder. There are a number of research results on perfect matching problem, see \cite{survy for hypergraph matching} for a survey.

For complete $k$-graphs and related, the research focus on the case $k=3$. Let $K_4^{3}$ be the complete 3-graph on four vertices, and $K_4 ^{3}-\ell e$  be the 3-graphs obtained from $K_4 ^{3}$ by deleting $\ell$ edges. K\"uhn and Osthus \cite{t_2 for K4-2e} showed that $t_2(n,K_4^{3}-2e)=(1/4+o(1))n$, and Czygrinow, DeBiasio and Nagle \cite{exact t_2 for K4-2e} determined its exact value for large $n$. Lo and Markstr\"om~\cite{t_2 for k4-e} proved that $t_2(n,K_4^{(3)}-e)=(1/2+o(1))n$ and the exact value of $t_2(n,K_4^{(3)}-e)$ was determined for large $n$ by Han, Lo, Treglown and Zhao \cite{nonextremal} recently. Lo and Markstr\"om \cite{absorption} also proved that $t_2(n,K_4^{(3)})=(3/4+o(1))n$, and
%independently ({\color{blue} They used completely differently methods and \cite{t_ for loose cycles} also wrote "independently"})
the exact value of $t_2(n,K_4^{(3)})$ was determined for large $n$ by Keevash and Mycroft \cite{exact t_2 for K4}.

A $(k,\ell)$-cycle $C_s^{(k,\ell)}$ is a $k$-graph on $s$ vertices so that whose vertices can be ordered cyclically in such a way that the edges are sets of consecutive $k$ vertices and every two consecutive edges share exactly $\ell$ vertices.
%Note that by definition, $s$ must be divisible by $k-1$ and at least $3k-3$.
Gao and Han \cite{C_6} and Czygrinow \cite{C_s3} determined the exact value of $t_2(n,C_6^{(3,1)})$ and $t_2(n,C_s^{(3,1)})(s\geq 6)$, respectively, and Gao, Han and Zhao \cite{t_ for loose cycles} determined $t_{k-1}(n,C_s^{(k,1)})$ for $k\ge 4$.
%A $k$-uniform tight cycle $\tilde{C}_s^k$ is an $s$-vertex $k$-graph whose vertices can be ordered cyclically such that every $k$ consecutive vertices under this ordering form an edge.
Han, Lo, and Sanhueza-Matamala \cite{t_ for tight cycles} proved $t_{k-1}(n, {C}_s^{(k,k-1)})\leq(1/2+1/(2s)+o(1))n$ where $k\geq3$ and $s\geq5k^2$ and this bound is asymptotically best possible for infinitely many pairs of $s$ and $k$.

In the study of tiling problems, another family of hypergraphs which was well studied are $k$-partite $k$-graphs.  A $k$-graph $F$ on vertex set $V$ is said to be {\it $k$-partite} if $V$ can be partitioned into vertex classes $V_1,\ldots, V_k$ so that for any $e\in F$ and $1\leq j\leq k$ we have $|e\cap V_j|$ =1. The partition $V_1,\ldots, V_k$ of $V$ is called a {\it $k$-partite realisation} of $V$.
%A $k$-partite realisation $\chi$ of $F$ is a partition of  Note in particular that we must have $|V_j|\geq1$ for every $1\leq j\leq k$.
Define $$\mathcal{S}(F):=\bigcup_{\chi}\{|V_1|,\ldots, |V_k|\} \mbox{ and } \mathcal{D}(F):=\bigcup_{\chi}\{|V_i|-|V_j| : i,j\in[k]\},$$ where in each case the union is taken over all $k$-partite realisations $\chi=\{V_1,\ldots, V_k\}$ of $V$.
%into vertex classes $U_1,\ldots, U_k$ of $F$.
The {\em greatest common divisor} of $F$, denoted by $\gcd(F)$, is then defined to be the {greatest common divisor} of the set $\mathcal{D}(F)$ (if $\mathcal{D}(F)=\{0\}$ then $\gcd(F)$is undefined). The {\em smallest class ratio} of $F$, denoted by $\sigma(F)$, is defined by
\begin{equation*}
\sigma(F):= \min_{S\in\mathcal{S}(F)}\frac{S}{|V(F)|}.
\end{equation*}
Note in particular that $\sigma(F)\leq1/k$, with equality if and only if $|V_1|=|V_2|=\ldots=|V_k|$ for any $k$-partite realisation $\chi=\{V_1, V_2,\ldots, V_k\}$ of $F$.
A {\it complete $k$-partite $k$-graph} with vertex classes $V_1,\ldots, V_k$ is a $k$-graph on $V=V_1\cup\ldots\cup V_k$ and edge set $E=\{ e : |e\cap V_i|=1 \mbox{ for each } i\in [1,k] \}$.
%in which every set $e\subset V$ with $|e\cap V_j|=1$ for each $V_1,\ldots, V_k$ is an edge.
Observe that a complete $k$-partite $k$-graph has only one $k$-partite realisation up to permutations of the vertex classes $V_1,\ldots, V_k$. Hence, we write $K^{k}(V_1,\ldots,V_k)$ for a {complete $k$-partite $k$-graph} with vertex classes $V_1,\ldots, V_k$ and if the sizes of $V_i$ are emphasized, we write $K^k(|V_1|,\ldots,|V_k|)$ for $K^k(V_1,\ldots,V_k)$,  if $|V_1|=\ldots=|V_k|=m$ we write $K^{k}(m)$ for $K^{k}(V_1,\ldots,V_k)$ and call $K^{k}(m)$ the {\it balanced complete $k$-partite $k$-graph}.  Mycroft \cite{complete k-partite} proved a general result on tiling $k$-partite $k$-graphs.

\begin{thm}[Theorem 1.1, 1.2, 1.3 in \cite{complete k-partite}]\label{THM: Mycroft}
Let $H$ be a $k$-partite $k$-graph. Then for any $\alpha>0$ there exists $n_0$ such that if $G$ is a $k$-graph on $n\ge n_0$ vertices for which $|V(H)|$ divides $n$ and
\begin{equation} \label{e1}
\delta_{k-1}(G)\ge\left\{\begin{array}{ll}
                     n/2+\alpha n  & \mbox{if } {\mathcal{S}(F)=\{1\} \mbox{ or } \gcd(S(F))>1;}\\
                     \sigma(F)n+\alpha n & \mbox{if } {\gcd(F)=1;}\\
                      \max\{\sigma(F)n, \frac np\}+\alpha n& \mbox{if } {\gcd(\mathcal{S}(F))=1 \mbox{ and } \gcd(F) >1,}
                   \end{array}
           \right.
\end{equation}
then $G$ contains a perfect $F$-tiling, where $p$ is the smallest prime factor of $\gcd(F)$.  Moreover, the  equality holds in (\ref{e1}) for a large class of $k$-partite $k$-graphs including all complete $k$-partite $k$-graphs.
\end{thm}

 Furthermore, Mycroft also conjectured that the error terms in (\ref{e1}) can be replaced by a (sufficiently large) constant that depends only on $F$.

\begin{conj}[\cite{complete k-partite}]\label{conjecture}
Let $F$ be a $k$-partite $k$-graph. Then there exists a constant $C=C(F)$ such that the error terms in (\ref{e1}) can be replaced by $C$.
\end{conj}

Gao, Han and Zhao \cite{t_ for loose cycles} improved the error term for complete $k$-partite $k$-graphs $F=K^k(a_1,\ldots,a_{k-1}, a_k)$ with $\gcd(F)=1$ and disproved Conjecture \ref{conjecture} for all complete $k$-partite $k$-graphs $F$ with $\gcd(F)=1$ and $a_{k-1}\geq2$. Han, Zang, and Zhao \cite{lattice-based} determined $t_1(n, K)$ asymptotically for all complete 3-partite 3-graphs $K$. In this paper, we focus on balanced complete 3-partite 3-graphs. One of our main results is the following.

\begin{thm}\label{mainK_m}
Let $m\geq2$ be an integer. There exists an integer $n_0\in\mathbb{N}$ such that the following holds. Suppose that $H$ is a 3-graph on $n\geq n_0$ vertices with $n\in 3m\mathbb{N}$. If $\delta_2(H)\geq n/2+m^{\frac1m}n^{1-\frac{1}{m}}$ then $G$ contains a $K^3(m)$-factor.
\end{thm}
For $K^3(2)$, we show that the lower bound of $\delta_2(G)$ is tight up to a factor.
\begin{prop}\label{lowerbound}
There exists an integer $n_1\in\mathbb{N}$. For every $n\geq n_1$, there exists a  $3$-graph $H$ on $n$ vertices with $\delta_2(H)\geq n/2+\sqrt{2n}/5-3$ containing no $K^3(2)$-factor.
\end{prop}

Clearly, Theorem~\ref{mainK_m} improves the error term $\alpha n$ in (\ref{e1}) to $Cn^{1-1/m}$ when $F=K^3(m)$, and Proposition~\ref{lowerbound} shows that the error term $C\sqrt{n}$ can not be replaced by a constant for $F=K^3(2)$ {and henceforth for $F=K^3(2m)$, which gives another family of counterexamples} for Conjecture \ref{conjecture} (Gao, Han and Zhao~\cite{t_ for loose cycles} gave a family of counterexample for Conjecture \ref{conjecture} when $F=K^k(a_1,\ldots,a_{k-1}, a_k)$ with $\gcd(F)=1$ and {$a_{k-1}\ge 2$.} Note that $\gcd(F)$ is undefined for $F=K^3(m)$. Hence our counterexample is  different from the one given in~\cite{t_ for loose cycles}).

Given integer $k$, let $\mathcal{C}_4^k$ be the family of $k$-graphs which contains four distinct edges $A$, $B$, $C$, $D$ with $A\cup B=C\cup D$ and $A\cap B=C\cap D=\emptyset$, which was first introduced by Erd\H{o}s~\cite{generalized 4-cycles}, and also is called the generalized 4-cycles.
%is a family of $k$-graphs which contains four distinct edges $A$, $B$, $C$, $D$ with $A\cup B=C\cup D$ and $A\cap B=C\cap D=\emptyset$.
For $k=2$ or 3,  we write $C_4^k$ for $\mathcal{C}_4^k$ instead because there is only one graph, up to isomorphism, in $\mathcal{C}_4^k$ in these cases. Note that $C_4^3$ is a supported subgraph of $K^3(2)$.
%the exact value of $t_2(n,C_4^3)$ for sufficiently large $n$.

Let $X_1,X_2,\ldots,X_t$ be $t$ pairwise disjoint sets of size $k-1$ and let $Y$ be a set of $s$ elements disjoint from $\cup_{i\in[t]}X_i$. Define $K^k_{s,t}$ be the $k$-graph with vertex set $(\cup_{i\in [t]}X_i)\cup Y$ and edge set $\{X_i\cup \{y\}: i\in[t], y\in Y\}$. In \cite{bipartite Turan}, Mubayi and Verstra\"ete investigated the Tur\'an number of $K^3_{s,t}$. We show that Conjecture~\ref{conjecture} is valid for $K^3_{m,m}$, in particular for generalized 4-cycle since $K^3_{2,2}=C_4^3$. More precisely, we prove the following theorem.

\begin{thm}\label{mainC4}
For any integer $m$, there exists an integer $N$ such that for all $n\in3m\mathbb{N}$ and $n\ge {N}$, each 3-graph $H$ on $n$ vertices with
\begin{equation}\label{codegree condition}
\delta_2(H)\geq\left\{\begin{array}{ll}
                     \lfloor n/2\rfloor-1, & \mbox{if } {n\equiv1\pmod4}\\
                     \lceil n/2\rceil-1, & \mbox{otherwise.}\\
                   \end{array}
           \right.
\end{equation}
contains a $K^3_{m,m}$-factor.
\end{thm}

To show the lower bound in Theorem~\ref{mainC4} is tight, we give a construction of extremal 3-graph for $K^3_{m,m}$.
\begin{cons}\label{CONS: 1.1}
Given two disjoint sets $A,B$, let $\mathcal{B}[A,B]$ be  a 3-graph with vertex set $A\cup B$ and edge set $E=\{e : |e|=3 \mbox{ and } |e\cap A|=1 \mbox{ or } 3\}$.
\end{cons}

Clearly, $\delta_2(\mathcal{B}[A,B])=\min\{|A|-2,|B|-1\}$, and each copy of $K_{m,m}^3$ intersects $B$ with an even number of vertices and hence $\mathcal{B}[A,B]$ does not contain a $K_{m,m}^3$-factor provided that $|B|$ is odd. Now, suppose that $n\in 3m\mathbb{N}$. {Choose $|A|=n/2+1, |B|=n/2-1$ if $n\equiv0 \pmod{4}$;  $|A|=\lfloor n/2\rfloor, |B|=\lceil n/2\rceil$ if $n\equiv1 \pmod{4}$; $|A|=|B|=n/2$ if $n\equiv2 \pmod{4}$; and $ |A|=\lceil n/2\rceil, |B|=\lfloor n/2\rfloor$ if $n\equiv 3 \pmod{4}$.} We have $\delta_2(\mathcal{B}[A,B])=\lfloor n/2\rfloor-2$ if $n\equiv1 \pmod{4}$, and $\delta_2(\mathcal{B}[A,B])=\lceil n/2\rceil-2$, otherwise. But $\mathcal{B}[A,B]$ does not contain a $K_{m,m}^3$-factor.
The extremal 3-graph constructed here implies that (\ref{codegree condition}) is tight.
%If $n\equiv6~mod~12$ and $|A|=|B|=\frac{n}{2}$, again we have $\delta_2(\mathcal{B}[A,B])=n/-2$ but $\mathcal[A,B]$ does not contain a $C_4$-factor. The next theorem shows that this bound is sharp.

In the following we give some notation used in this paper.  For an $r$-graph $H=(V,E)$ and a vertex set $U\subseteq V$, write $H[U]$ for the subgraph of $H$ induced by $U$ and $\binom{U}{k}$ for  the set of all subsets of size $k$ of $U$.  For an $S\subseteq V$, the {\it neighbourhood} of $S$, denoted by $N_H(S)$ or $N(S)$ if there is no confusion from the context, is the set of subsets $T\subseteq V$ such that $S\cup T\in E(H)$, the {\it link graph} of $S$, denoted by $H_S$,  is the $(r-|S|)$-graph with vertex set $V(H)\setminus S$ and edge set $N_H(S)$. For a 3-graph $H=(V,E)$ and $u,v,w\in V$, we write   $uv$ and $uvw$ for the sets $\{u,v\}$ and $\{u,v,w\}$, respectively. Let $V_1, \ldots, V_r$ be a partition of $V(H)$. An edge $e=v_1v_2v_3$ is of type $V_{i_1} V_{i_2} V_{i_3}$ if $v_j\in V_{i_j}$ for $j\in[3]$ and $i_j\in [r]$, write $E(V_{i_1}V_{i_2}V_{i_3})$ for the set of edges of type $V_{i_1}V_{i_2}V_{i_3}$ and $e(V_{i_1}V_{i_2}V_{i_3})=|E(V_{i_1}V_{i_2}V_{i_3})|$.
%Let $V(H)=A\cup B$ be a partition, and an $t$-vertex $F$ be a subgraph of $H$. We call $F$ of type $(a,t-a)$ if $|V(F)\cap A|=a$ and $|V(F)\cap B|=t-b$.
A subgraph $F$ of $H$ is said to be of type $(t_1,\ldots, t_r)$ if $|V(F)\cap V_i|=t_i$ for each $i\in [r]$.

\section{Lemmas and proofs of main results}
%In this section we will prove the main theorems.

To show Proposition~\ref{lowerbound}, we first revisit a construction of $K^{k}(1,\ldots,1,2,t+1)$-free ($t\geq 1$) $k$-graph $G$ with $e(G)\sim\frac{\sqrt{t}}{k!}n^{k-\frac{1}{2}}$ edges given by Mubayi~\cite{lower bound}.
%for Tur\'an number of $k$-partite $k$-graphs. Specifically, Mubayi give a construction of $K^{(k)}(1,\ldots,1,2,t)$-free ($t\geq2$) $k$-graph $G$ with $e(G)\sim\frac{\sqrt{t}}{k!}n^{k-\frac{1}{2}}$ edges.
We only need the special case that $k=3$ and $t=1$. Let $q$ be a prime power and $\mathbf{F}_q$ be the $q$-element finite field.
\begin{cons}[\cite{lower bound}]\label{CONS: 2.1}
Let $G_q$ be a 3-graph  with vertex set $V(G_q)=(\mathbf{F}_q\setminus\{0\})\times(\mathbf{F}_q\setminus\{0\})$, a 3-elements set $\{(a_i, a_i') : i\in[3]\}$ forms an edge in $G_q$ if and only if
\begin{equation*}
\prod_{i\in[3]}a_i+\prod_{i\in[3]}a_i'={\bf 1}_F.
\end{equation*}
\end{cons}
As shown in \cite{lower bound}, $G_q$ is $K^3(1,2,2)$-free and $\delta_2(G_q)\geq q-3$.

\begin{cons}\label{CONS: 2.2}
Let $G_q'$ be a copy of $G_q$. Define $H_q$ to be a 3-graph  with vertex set $V(G_q)\cup V(G_q')$, each edge of $H_q$ intersect $V(G_q)$ in precisely two vertices and a 3-elements set $\{(a_i, a_i') : i\in[3]\}$ with $(a_i, a_i')\in V(G_q)$ for $i=1,2$ and $(a_3, a_3')\in  V(G_q')$ forms an edge in $H_q$ if and only if
\begin{equation*}
\prod_{i\in[3]}a_i+\prod_{i\in[3]}a_i'={\bf 1}_F.
\end{equation*}
For convenience, we use ordered triple $(a,b,c)$ denote an edge of $H_q$ with $a,b\in V(G_q)$ and $c\in V(G_q')$.
\end{cons}

\noindent{\bf Remark.} By the constructions of $G_q$ and $H_q$, we know that an edge $e=abc\in E(G_q)$ corresponds to three edges  $e_1=(a,b,c), e_2=(a,c,b), e_3=(b,c, a)$ in $H_q$, and $H_q$ possibly contains some edges of the form $(a,b,a)$ or $(a,b,b)$. The following fact shows that $H_q$ inherits the property from $G_q$.
\begin{fact}\label{FACT: f2.1}
The 3-graph $H_q$ is $K^3(1,2,2)$-free and $\delta_2(H_q)\ge q-3$.
\end{fact}
\begin{proof}
%\noindent{\it\bf Proof of the Fact \ref{FACT: f2.1}:}
Clearly,  $\delta_2(H_q)\ge q-3$ since $G_q$ is a subgraph of $H_q$. We show that $H_q$ is also $K^3(1,2,2)$-free.
As shown in \cite{lower bound}, for $(p_1, q_1), (p_2, q_2)\in(\mathbf{F}_q\setminus\{0\}\times(\mathbf{F}_q\setminus\{0\})$, the equation system
\begin{equation}\label{K-free}
\left\{
  \begin{array}{l}
   p_1x + p_1'y = 1_F \\
   p_2x + p_2'y = 1_F
  \end{array}
\right.
\end{equation}
has at most one solution $(x, y)$ if $(p_1, p_1')\ne (p_2, p_2')$. Suppose that $H_q$ contains a copy of $K^3(1,2,2)$, say $K^3(\{a\},\{b_1,b_2\},\{c_1,c_2\})$. Let $a=(u,u')$, $b_1=(v_1,v_1')$ and $b_2=(v_2,v_2')$. Without loss of generality, we may assume $a, b_1, b_2\in V(G_q)$.
%{\color{blue} Obviously, both of $b_1$ and $b_2$ in $V(G_q)$ or $V(G_q')$. (All edges are of the same type, and just consider edges $ab_1c_1$ and $ab_2c_1$.)}
Now let $p_1=uv_1$, $p_1'=u'v_1'$, $p_2=uv_2$,  and $p_2'=u'v_2'$.  Since $(v_1,v_1')\neq(v_2,v_2')$, we have $(p_1, p_1')\ne (p_2, p_2')$. So the equation system (\ref{K-free}) has at most one solution, this is a contradiction to $K^3(\{a\},\{b_1,b_2\},\{c_1,c_2\})\subseteq H_q$. \quad %\rule{1mm}{2mm}
\end{proof}
\medskip

\noindent{\it\bf Proof of Proposition \ref{lowerbound}:} { For sufficiently large $n$, without loss of generality, we may assume $n\in 6\mathbb{N}$, choose an odd prime power $q$ and $n_0=(q-1)^2$ such that $n/2+2/5\sqrt{n/2}\le n_0\le n/2+1/2\sqrt{n/2}$.} Let $\mathbf{F}_q$ be the $q$-element finite field and let $A$ and $B$ be the sets obtained by deleting any one element and  $2n_0-n-1$ elements from $(\mathbf{F}_q\setminus\{0\})\times(\mathbf{F}_q\setminus\{0\})$, respectively. Then $|A|=n_0-1$ and $|B|=n-n_0+1$, both of  them are odd.
%We refer the elements in $A$ and $B$ as $(a_i,a_j)$ and $(b_i, b_j)$ respectively. {\color{red} Define $H'$ be a 3-graph with vertex set $A\cup B$, and $(a_{i_1},a_{j_1})$, $(a_{i_2},a_{j_2})$ and $(b_i, b_j)$ form an edge in $H'$ if and only if $a_{i_1}a_{i_2}b_i+a_{j_1}a_{j_2}b_j={\bf 1}_F$. Obviously, $H'$ is $K^3(1,2,2)$-free and $d_{H'}(ab)\geq q-4$ for all $a\in A, b\in B$.}
Let $H'$ be the subgraph of $H_q$ induced by $A\cup B$ with $A\subset V(G_q)$ and $B\subset V(G_q')$. By Fact~\ref{FACT: f2.1}, $H'$ is $K^3(1,2,2)$-free and $d_{H'}(ab)\geq q-4$ for all $a\in A, b\in B$.
Let $H=\mathcal{B}[A,B]\cup H'$. Then {$\delta_2(H)\geq \min\{|A|-2,|B|-1+\sqrt{n_0}-3\}\geq n/2+2/5\sqrt{n/2}-3$.
We claim that $H$ does not contain a $K^3(2)$-factor.  Suppose to the contrary that $H$ contains a $K^3(2)$-factor. Since $|A|$ is odd, $H$ must contain a copy of $K^3(2)$ such that $|V(K^3(2))\cap A|$ is odd. Such a copy of $K^3(2)$ must be of type $(5,1)$ or $(3,3)$. Note that {copies of $K^{{3}}(2)$ in $\mathcal{B}[A,B]$ must intersect  $A$  in an even number of vertices}. It is an easy task to check that a copy of $K^3(2)$ of type $(5,1)$ or $(3,3)$ forces a copy of $K^3(1,2,2)$ in $H'$, a contradiction. \quad \rule{1mm}{2mm}

\medskip

The proof of Theorems \ref{mainK_m} and \ref{mainC4} are separated into {\it non-extremal} case and {\it extremal} case. For the non-extremal case, we use the standard absorbing method, which has been introduced by R\"{o}dl, Ruci\'nski and Szemer\'edi in \cite{original abs} and widely used in different research papers for example in \cite{exact t_2 for K4-2e,lattice-based,absorption}.

Roughly speaking, our proof follows two steps: first, we use an "absorbing lemma" to find a small absorbing set $W\subset V(H)$ with the property that given any "sufficiently small" set $U\subset V(H)\setminus W$, both $H[W]$ and $H[W\cup U]$ contain $K^3(m)$-factors; second, we use an "almost tiling lemma" to find a $K^3(m)$-tilling in $H\setminus W$ that covers all but at most $o(n)$ vertices. The first step will be completed in Lemma~\ref{absorption} and the second step has been done by an almost tiling lemma given by Mycroft in~\cite{complete k-partite}, we restate it in Lemma~\ref{almost packing}.

Given $\gamma>0$, $H$ and $G$ are two 3-graphs on the same vertex set $V$. We say that $H$ {\it $\gamma$-contains} $G$ if $|E(G)\setminus E(H)|\leq\gamma|V|^3$, and $H$ is called {\it $\gamma$-extremal} if there is a balanced partition of $V=A\cup B$ such that $|A|\le |B|$ and  $H$ $\gamma$-contains $\mathcal{B}[A,B]$.
%Given $\gamma>0$, we call a 3-graph $H=(V,E)$ on $n$ vertices $\gamma$-extremal if there is a partition of $V=A\cup B$ such that $|A|=\lfloor n/2\rfloor$, $|B|=\lceil n/2\rceil$ and $H$ $\gamma$-contains $\mathcal{B}[A,B]$.

\begin{lem}[Absorption lemma]\label{absorption}
Let $0<\epsilon_2\ll\epsilon_1\ll\gamma\ll1$ and $m$ be an positive integer. Suppose that $H$ is a 3-graph of order $n$ with $\delta_2(H)\geq(1/2-\gamma)n$. If $H$ is not $3\gamma$-extremal, then there exists a set $W\subset V(H)$ with $|W|\le\epsilon_1 n$ and $|W|\in 3m\mathbb{N}$, so that for any $U\subset V(H)\setminus W$ with $|U|\leq\epsilon_2n$ and $|U|\in 3m\mathbb{N}$, both $H[W]$ and $H[U\cup W]$ contain $K^3(m)$-factors.
%{\color{red} WE DON'T USE $H[W]$ CONTAINS $K^3(m)$-FACTOR, CAN WE DELETE IT? X0512 }
\end{lem}

\begin{lem}[Almost tiling lemma, Lemma 1.5 in \cite{complete k-partite}]\label{almost packing}
Let $K$ be a $k$-partite $k$-graph. Then there exists a constant $C=C(K)$ such that for any $\alpha>0$ there exists an integer $n_0=n_0(K,\alpha)$ with the property that every $k$-graph $H$ on $n\geq n_0$ vertices with $\delta_{k-1}(H)\geq(\sigma(K)+\alpha)n$ admits a $K$-tiling covering all but at most $C$ vertices of $H$.
\end{lem}

Lemmas \ref{extremalK_m} and \ref{extremalC_4} deal with the extremal case for $K^3(m)$ and $K_{m,m}^3$, respectively.

\begin{lem}\label{extremalK_m}
Let $m\geq2$ be an integer. There exist $\gamma>0$ and $n_0\in\mathbb{N}$ such that the following holds. Suppose that $H$ is a 3-graph on $n\geq n_0$ vertices with $\delta_2(H)\geq n/2+m^\frac{1}{m}n^{1-\frac{1}{m}}$, $n\in 3m\mathbb{N}$. If  $H$ is $\gamma$-extremal, then $H$ contains a $K^3(m)$-factor.
\end{lem}

\begin{lem}\label{extremalC_4}
There exist $\gamma>0$ and $n_0\in\mathbb{N}$ such that the following holds. Suppose that $H$ is a 3-graph on $n\geq n_0$ vertices with $\delta_2(H)$ satisfying (\ref{codegree condition}), where $n\in 3m\mathbb{N}$. If $H$ is $\gamma$-extremal, then $H$ contains a $K_{m,m}^3$-factor.
\end{lem}

\medskip

\noindent{\it\bf  Proof of Theorems \ref{mainK_m} and \ref{mainC4}:} Let $0<\alpha\ll1$ and $1/n\ll\epsilon_2\ll\epsilon_1\ll\gamma\ll 1$ with $n\in3m\mathbb{N}$. Let $H$ be a 3-graph of order $n$ with $\delta_2(H)\geq n/2+m^\frac{1}{m}n^{1-\frac{1}{m}}$ (resp. $\delta_2(H)$ satisfying (\ref{codegree condition})).

I. $H$ is $3\gamma$-extremal. Then, by Lemma \ref{extremalK_m},  $H$ contains a $K^3(m)$-factor (resp.  $K_{m,m}^3$-factor by Lemma~\ref{extremalC_4}).

II. $H$ is not $3\gamma$-extremal.
Note that $K_{m,m}^3$ is a spanning subgraph of $K^3(m)$ by the definition of $K_{m,m}^3$. If $H$ has a $K^3(m)$-factor then it also contains a $K_{m,m}^3$-factor.   By Lemma \ref{absorption}, we can choose  an absorbing set $W\subset V(H)$ with $|W|\le\epsilon_1 n$ and $|W|\in 3m \mathbb{N}$ so that for any $U\subset V(H)\setminus W$ with $|U|\leq\epsilon_2 n$ and $|U|\in 3m \mathbb{N}$, both $H[W]$ and $H[U\cup W]$ contain $K^3(m)$-factors.
Let $H'$ be the 3-graph obtained from $H$ by deleting the vertices of $W$. Then $|V(H')|=n'\geq(1-\epsilon_1)n$ and $\delta_2(H')\geq n/2-1-\epsilon_1 n\geq (1/3+\alpha)n'$. {Note that $\sigma(K^3(m))=1/3$.} The codegree condition in Lemma~\ref{almost packing} for $H'$ and $K^3(m)$ is satisfied. By Lemma \ref{almost packing}, $H'$ contains a $K^3(m)$-tiling $M_1$ covering all but at most $C$ vertices. Let $U=V(H')\setminus V(M_1)$. Then $|U|=n-|W|-|V(M_1)|\in 3m\mathbb{N}$ and $|U| \leq C\leq \epsilon_2 n$. Hence $H[U\cup W]$ contains a $K^3(m)$-factor $M_2$. Then $M_1\cup M_2$ is a $K^3(m)$-factor in $H$. We are done. \quad \rule{1mm}{2mm}

\medskip

The rest of the paper is organized as follows. In Section 3, we give the proof of the absorption lemma used in the paper, i.e. Lemma \ref{absorption}, and in Section 4, we deal with the extremal case, i.e. we prove Lemmas~\ref{extremalK_m} and~\ref{extremalC_4}.

\section{Absorption lemma}
To prove the absorption lemma, we need some preliminaries.
Let $H=(V, E)$ be a $k$-graph of order $n$, and $F$ be a $k$-graph of order $t$. Given an integer $i\geq1$, a constant $\eta>0$, and two vertices $x, y\in V$, a vertex set $S\subset V$ is called an {\it $(x,y)$-connector} of length $i$ with respect to $F$ if $S\cap\{x,y\}=\emptyset$, $|S|=ti-1$ and both $H[S\cup\{x\}]$ and $H[S\cup\{y\}]$ contain $F$-factors.
Two vertices $x$ and $y$ are called {\it $(i,\eta)$-close} with respect to $F$ if there exist at least $\eta n^{ti-1}$ $(x, y)$-connectors of length $i$ with respect to $F$ in $H$. Let
$$\tilde{N}_{F,i,\eta}(x)=\{y : x \mbox{ and } y \mbox{ are } (i,\eta)\mbox{-close with respect to } F\}.$$
 A subset $U\subset V$ is said to be $(F, i,\eta)$-closed in $H$ if any two vertices in $U$ are $(i,\eta)$-close with respect to $F$. If $V$ is $(F, i,\eta)$-closed in $H$ then we simply say that $H$ is $(F, i, \eta)$-closed.

The following lemma given by Lo and Markstr\"{o}m \cite{absorption} referred to as absorption lemma provides a absorbing set for any sufficiently small vertex set if $H$ is $(F, i, \eta)$-closed.

\begin{lem}[Lemma 1.1 in \cite{absorption}]\label{abp}
Let $t$ and $i$ be positive integers and $\eta>0$. Then there exist $\eta_1,\eta_2$ such that $0<\eta_2\ll\eta_1\ll\eta$ and an integer $n_0=n_0(i,\eta)$ satisfying the following: Suppose that $F$ is a $k$-graph of order $t$ and $H$ is an $(F,i,\eta)$-closed $k$-graph of order $n\geq n_0$. Then there exists a vertex subset $U\subset V(H)$ of size at most $\eta_1 n$ with {$|U|\in t\mathbb{Z}$} such that, for every vertex set $W\subset V\setminus U$ of size at most $\eta_2 n$ with {$|W|\in t\mathbb{Z}$}, both $H[U]$ and $H[U\cup W]$ contain $F$-factors.

%there exists an $F$-factor in $H[U\cup W]$ for every vertex set $W\subset V\setminus U$ of size $|W|\leq\eta_2 n$ with $|W|\in\mathbb{Z}$.
\end{lem}

Lemma \ref{closed pair} also given in \cite{absorption} allows us to find close pairs with respect to a $k$-partite $k$-graph $F$.

\begin{lem}[Lemma 4.2 in \cite{absorption}]\label{closed pair}
Let $k\geq2$ be an integer and $\alpha>0$. Given a $k$-partite $k$-graph $F$, there exist a constant $\eta_0=\eta_0(k,F,\alpha)>0$ and an integer $n_0=n_0(k,F,\alpha)$ such that the following holds: Let $H$ be a $k$-graph of order $n\geq n_0$ and $x, y\in V(H)$. If
$$\left|\{ S \,|\, S\in N(x)\cap N(y)  \mbox{ with } |N(S)|\geq\alpha n\}\right|\ge\alpha\binom{n}{k-1},$$ then $x$ and $y$ are $(F,1,\eta)$-close for all $0<\eta\leq\eta_0$.
\end{lem}

The following lemma in \cite{partition} gives us a partition of $V(H)$  with bounded number of parts such that each of them is closed with respect to $F$.

\begin{lem}[Lemma 6.3 in \cite{partition}]\label{partition}
Given $\delta>0$, integers $c, k, t\geq2$ and $0<\eta\ll1/c,\delta, 1/t$, there exists a constant $\eta'>0$ such that the following holds for all sufficiently large $n$: Let $F$ be a $k$-graph on $t$ vertices. Assume a $k$-graph $H$ on $n$ vertices satisfies that $|\tilde{N}_{F,1,\eta}(v)|\geq \delta n$ for any $v\in V(H)$ and every set of $c+1$ vertices in $V(H)$ contains two vertices that are $(F,1,\eta)$-close. Then we can find a partition of $V(H)$ into $V_1,\ldots,V_r$ with $r\leq \min\{c, 1/\delta'\}$ such that for any $i\in[r], |V_i|\geq(\delta-\eta)n$ and $V_i$ is $(F, 2^{c-1},\eta')$-closed in $H$.
\end{lem}

Actually here we use a variant absorbing method which is so-called lattice-based absorption developed by Han in~\cite{Han-Trans17}. The following definitions are introduced by Keevash and Mycroft~\cite{exact t_2 for K4}.
Given a $k$-graph $H=(V, E)$ and a partition $\mathcal{P}=\{V_1,\ldots,\,V_r\}$ of $V$, the {\it index vector} $\mathbf{i}_\mathcal{P}(S)$ of a subset $S\subset V$ with respect to $\mathcal{P}$ is the vector whose $i$-th coordinate is the size of the intersection of $S$ and $V_i$. A vector $\mathbf{v}\in\mathbb{Z}^r$ is called an {\it $s$-vector} if all its coordinates are nonnegative and their sum equals to $s$.
Given a $k$-graph $F$ of order $t$ and $\mu>0$, a $t$-vector $\mathbf{v}$ is called a {\it $\mu$-robust $F$-vector} if there are at least $\mu n^t$ copies $F'$ of $F$ in $H$ satisfying $\mathbf{i}_\mathcal{P}(V(F'))$=$\mathbf{v}$. Let $I_{\mathcal{P},F}^{\mu}(H)$ be the set of all $\mu$-robust $F$-vectors and $L_{\mathcal{P},F}^{\mu}(H)$ be the lattice (i.e., the additive subgroup) generated by $I_{\mathcal{P},F}^{\mu}(H)$.
For $j\in[r]$, let $\mathbf{u}_j\in \mathbb{Z}^r$ be the $j$-th unit vector, namely, $\mathbf{u}_j$ has 1 on the $j$-th coordinate and 0 on other coordinates. A transferral is a vector of the form $\mathbf{u}_i-\mathbf{u}_j$ for some distinct $i, j\in[r]$.

The following lemma in \cite{lattice-based} states that if $L_{\mathcal{P},F}^{\mu}(H)$ contains all transferrals then $H$ is closed.

\begin{lem}[Lemma 3.9 in~\cite{lattice-based}]\label{lattice}
Let $i_0, k, r_0>0$ be integers and let $F$ be a $k$-graph on $t$ vertices. Given constants $\epsilon,\eta,\mu>0$, there exist $\eta'>0$ and an integer $i'_0\geq0$ such that the following holds for sufficiently large $n$: Let $H$ be a $k$-graph on $n$ vertices with a partition $\mathcal{P}=\{V_1,  \ldots, V_r\}$ such that $r\leq r_0$ and for each $j\in[r], |V_j|\geq \epsilon n$ and $V_j$ is $(F, i_0, \eta)$-closed in $H$. If $\mathbf{u}_j-\mathbf{u}_\ell\in L_{\mathcal{P},F}^{\mu}(H)$ for all $1\leq j<\ell\leq r$, then $H$ is $(F, i'_0,\eta')$-closed.
\end{lem}

The following lemma helps us to count the number of copies of $K^3(m)$.
\begin{lem}[Corollary 2 in~\cite{Erdos-Simonovits-83} ]\label{supersaturation}
Let $F$ be a $k$-partite $k$-graph of order $t$. For every $\epsilon>0$, there exists a constant $\mu>0$ and an integer $n_0$ such that every $k$-graph $H$ of order $n\geq n_0$ with $e(H)>\epsilon n^k$ contains at least $\mu n^t$ copies of $F$.
\end{lem}

 We also need the following lemma from \cite{nonextremal}.

\begin{lem}[Lemma 3.3 in \cite{nonextremal}]\label{nonextremal}
Let $0<1/n\ll\gamma< 1/100$. Suppose that $H$ is a 3-graph of order $n$ with $\delta_2(H)\geq(1/2-\gamma)n$. Let $X,~Y$ be any bipartition of $V(H)$ with $|X|, |Y|\geq n/5$. If $H$ is not $3\gamma$-extremal, then $H$ contains  at least $\gamma^2n^3$ $XXY$-edges and at least $\gamma^2n^3$ $XYY$-edges.
\end{lem}

Now it is ready to give the proof of our absorption lemma, we restate it here.

\begin{lem}\label{absorption-c}
Let $0<\epsilon_2\ll\epsilon_1\ll\gamma\ll1$ and $m$ be an positive integer. Suppose that $H$ is a 3-graph of order $n$ with $\delta_2(H)\geq(1/2-\gamma)n$. If $H$ is not $3\gamma$-extremal, then there exists a set $W\subset V(H)$ with $|W|\le\epsilon_1 n$ and $|W|\in 3m\mathbb{N}$  so that for any $U\subset V(H)\setminus W$ with $|U|\leq\epsilon_2n$ and $|U|\in 3m\mathbb{N}$, both $H[W]$ and $H[U\cup W]$ contain $K^3(m)$-factors.
\end{lem}
\begin{proof}
%\noindent{\it Proof of the claim \ref{absorption}:}
Assume $\gamma$ is sufficiently small and let $\alpha=\gamma/3$. Let $F=K^3(m)$. If  we prove that $H$ is $(F, i, \eta)$-closed for some parameters  $i>0$ and $0<\eta\ll\gamma$, then by Lemma \ref{abp} with $t=3m$ we obtain the desired absorbing set. So in the following it is sufficient to show that $H$ is $(F, i, \eta)$-closed for some parameters  $i>0$ and $0<\eta\ll\gamma$.
%Suppose $\delta_2(H)\geq(1/2-\gamma)n$ and $H$ is not $3\gamma$-extremal.
%In the first step, we estimate the size of $\tilde{N}_{F,1,\eta}(v)$ for every $v\in V$ and show every set of three vertices contains two vertices which are $(F,1,\eta)$-closed to each other in $H$.

\begin{claim}\label{CLAIM: c1}
For each $v\in V(H)$ and some $0<\eta\ll\gamma$, $\tilde{N}_{F,1,\eta}(v)\geq (1/2-2\gamma)n$.
\end{claim}

\noindent{\it Proof of Claim~\ref{CLAIM: c1}:} Fix $v\in V(H)$, we have
\begin{equation}\label{e2}
|N(v)|\geq\frac{(1/2-\gamma)n(n-1)}2=(1/2-\gamma)\binom{n}{2}.
\end{equation}
Note that $|N(S)|\geq(1/2-\gamma)n\geq\alpha n$ for any 2-elements set $S\subseteq V(H)$. By Lemma~\ref{closed pair},  we have $u\in\tilde{N}_{F,1,\eta}(v)$ if $|N(v)\cap N(u)|\geq \alpha\binom{n}{2}$ for any $0<\eta\le \eta_0=\eta_0(k,F,\alpha)$. Let $G$ be a bipartite graph with partition classes $N(v)$ and $V(H)\setminus\{v\}$, and a 2-elements set $S\in N(v)$ and a vertex $w\in V(H)\setminus\{v\}$ are adjacent in $G$ if and only if $S\cup\{w\}\in E(H)$. Then we have
\begin{equation*}
e(G)=\sum_{S\in N(v)}d_G(S)=\sum_{S\in N(v)}(|N(S)|-1)<|\tilde{N}_{F,1,\eta}(v)|\cdot|N(v)|+n\cdot\alpha\binom{n}{2}.
\end{equation*}
Together with $|N(S)|\geq(1/2-\gamma)n$, we have
\begin{equation*}
|\tilde {N}_{F,1,\eta}(v)|\geq (1/2-\gamma)n-1-\frac{n\cdot\alpha\binom{n}{2}}{ (1/2-\gamma)\binom{n}{2}}{\geq (1/2-2\gamma)n}. \quad\quad \rule{1mm}{2mm}
\end{equation*}
%\quad \rule{1mm}{2mm}

Given any three vertices $x_1,x_2, x_3\in V(H)$, by~(\ref{e2}) and the inclusion-exclusion principle, we have
\begin{eqnarray*}
\sum_{1\le i<j\le 3}|N(x_i)\cap N(x_j)|&=& \sum_{i=1}^3|N(x_i)|-|\cup_{i=1}^3N(x_i)|+|\cap_{i=1}^3N(x_i)|\\
                                       &\ge& 3(1/2-\gamma){n\choose 2}-|\cup_{i=1}^3N(x_i)|+|\cap_{i=1}^3N(x_i)|\\
                                       &\geq& 3\alpha\binom{n}{2}+{n\choose 2}-|\cup_{i=1}^3N(x_i)|+|\cap_{i=1}^3N(x_i)|\\
                                       &\ge & 3\alpha\binom{n}{2}.
\end{eqnarray*}
By the pigeonhole principle, there exists at least one pair $x_i, x_j$ such that $|N(x_i)\cap N(x_j)|\geq \alpha\binom{n}{2}$, by Lemma~\ref{closed pair},  such a pair $x_i, x_j$  is $(F,1,\eta)$-close.

Now apply Lemma \ref{partition} to $F$ and $H$ with $\delta=(1/2-2\gamma)$, $c=2$ and $\eta\ll\gamma$, we have that there exist a constant $\eta'>0$ and  a partition $\mathcal{P}$ of $V$ with at most 2 parts such that each part has size at least $(1/2-3\gamma)n$ and is $(F,2,\eta')$-closed in $H$. If $|\mathcal{P}|=1$, then $H$ is  $(F,2,\eta')$-closed, as desired. So, we assume $|\mathcal{P}|=2$ and $\mathcal{P}=\{X, Y\}$.
Since $H$ is not $3\gamma$-extremal, by Lemma \ref{nonextremal}, we have both $e(XXY)$ and $e(XYY)$ are at least $\gamma^2n^3$.

Define $$E_0£º=\{xy:x\in X,y\in Y,\, d_X(xy)\geq\gamma^2 n, d_Y(xy)\geq\gamma^2 n\},$$
 $$E_1£º=\{xy:x\in X,y\in Y,\, d_X(xy)\geq\gamma^2 n, d_Y(xy)<\gamma^2 n\},$$
  and $$E_2£º=\{xy:x\in X,y\in Y,\, d_X(xy)<\gamma^2 n, d_Y(xy)\geq\gamma^2 n \}.$$
Then $E(K(X,Y))=E_0\cup E_1\cup E_2$. So $|E_i|\le e(K(X,Y))\le \frac {n^2}4$ for any $i\in \{0, 1,2\}$. By Lemma \ref{lattice}, to show that $H$ is closed it suffices to show $\mathbf{u}_1-\mathbf{u}_2\in L_{\mathcal{P},F}^{\mu}(H)$ for some $\mu$, or equivalently, we need to show that $H$ contains at least $\mu n^{3m}$ copies of $K^3(m)$ of types $(i,3m-i)$ and $(i+1,3m-i-1)$ for some $i$, respectively. We split the following proof into two cases depending on the size of $E_0$. The next claim deals with the case when $|E_0|$ is sufficiently large.

\begin{claim}\label{large E0}
There exists $\mu_1>0$ for any given integers $0\leq s,t\leq m$ with $s+t=m$ such that the following holds: If $|E_0|\geq\gamma^4 n^2$, then $H$ contains at least $\mu_1 n^{3m}$ copies of $K^3(m)$ of type $(m+s,m+t)$.
\end{claim}

\noindent{\it Proof of Claim \ref{large E0}:} Choose $0<\gamma_1\ll\gamma$.
 %and apply  Lemma \ref{supersaturation} with input $K^{3}(m)$ and $\gamma_1'$. Let  $\mu_0$ be the constant defined in Lemma \ref{supersaturation} and choose $0<\mu_1\ll\mu_0$.
%where $\mu_0$ is defined by lemma \ref{supersaturation} with input $K^{(2)}(m)$ and $\gamma_1$.
Given $x\in X, y\in Y$, let $k_{s,t}(xy)$ be the number of copies of $K^3({\{x\},\{y\}, M})$ where $|M\cap X|=s$ and $|M\cap Y|=t$. Given $S\subset X$ and  $T\subset Y$ with $|S|=s$ and $|T|=t$, let $G(S,T)$ be the bipartite graph with vertex set $X\cup Y$ and edge set consisting of all pairs $xy$ such that $x\in X, y\in Y$ and $xyz\in E(H)$ for all $z\in S\cup T$. Then by double counting we have
\begin{equation*}
\sum_{x\in X, y\in Y} k_{s,t}(xy)=\sum_{S\in\binom{X}{s},T\in\binom{Y}{t}} |G(S,T)|.
\end{equation*}
In the following,  we write LHS (resp. RHS) for left-hand side (resp. right-hand side) for short. Note that $|E_0|\geq\gamma^4 n^2$, we have
\begin{equation*}
LHS=\sum_{x\in X, y\in Y}\binom{d_X(xy)}{s}\binom{d_Y(xy)}{t}\geq |E_0|\binom{\gamma^2 n}{s}\binom{\gamma^2 n}{t}\geq \gamma_1 n^{m+2}.
\end{equation*}
On the other hand,
\begin{eqnarray*}
RHS&=&\left(\sum_{|G(S,T)|<\gamma_1 n^2}+\sum_{|G(S,T)|\geq\gamma_1 n^2}\right)|G(S,T)|\\
&\leq&\sum_{|G(S,T)|<\gamma_1 n^2}\gamma_1 n^2+\sum_{|G(S,T)|\geq\gamma_1 n^2}|X||Y|.
\end{eqnarray*}
So the number of pairs $(S,T)$ such that $|G(S,T)|\geq \gamma_1 n^2$ is at least
\begin{equation*}
\left(\gamma_1 n^{m+2}-\gamma_1 n^2\binom{|X|}{s}\binom{|Y|}{t}\right)/(|X||Y|)\geq \frac{1}{2}\gamma_1 n^{m+2}/(n^2/4)=2\gamma_1 n^m.
\end{equation*}
Fix such a pair $(S,T)$, by Lemma~\ref{supersaturation}, there is a positive constant $\mu_0$ such that $G(S,T)$  contains at least $\mu_0 n^{2m}$ copies of $K^{2}(m)$, which gives us at least $\mu_0 n^{2m}$ copies of $K^3(m)$ of type $(m+s,m+t)$. Choose $0<\mu_1\ll\mu_0$. Summing up all of such pairs $(S, T)$, we get at least $2\gamma_1 n^m\mu_0 n^{2m}\geq\mu_1 n^{3m}$ copies of $K^3(m)$ of type $(m+s,m+t)$ in $H$. \quad \rule{1mm}{2mm}

\begin{claim}\label{small E0}
Given integers $0\leq s,t\leq m$ with $s+t=m$, there exists $\mu_1'>0$ such that the following holds: If $|E_0|<\gamma^4 n^2$, then $H$ contains at least $\mu_1' n^{3m}$ copies of $K^3(m)$ of the same type either $(2m+s,t)$ or $(t,2m+s)$.
\end{claim}

\noindent{\it Proof of Claim~\ref{small E0}:}
%and $0<\mu_1'\ll\mu_0'$ where $\mu_0'$ is defined by lemma \ref{supersaturation} with input $K^{(2)}(m)$ and $\gamma_1'$.
Without loss of generality, assume that $|E_1|\leq |E_2|$.
%and we are going to show $H$ contains at least $\mu_1' n^{3m}$ copies of $K^3(m)$ of type $(2m+s,t)$.
First, we show $3\gamma^2 n^2\leq|E_1|\leq \frac 18n^2$. The upper bound is trivial by the assumption that $|E_1|\le |E_2|$. Now suppose that  $|E_1|<3\gamma^2n^2$. Then,  we have
\begin{eqnarray*}
e(XXY)&&=\frac 12\sum_{x\in X,y\in Y} d_X(xy)\\
&&<\frac 12\left(|E_0|\cdot|X|+|E_1|\cdot|X|+|E_2|\cdot\gamma^2n\right)\\
&&<\frac 12\left(\left(\gamma^4+3\gamma^2\right)n^2\cdot\left(\frac 12+3\gamma\right)n+\frac {n^2}4\cdot\gamma^2n\right)\\
&&<\gamma^2n^3,
\end{eqnarray*}
%ogether with $|E_0|<\gamma^4 n^2$,
%Since $H$ is not $3\gamma$-extremal, by lemma \ref{nonextremal} we have
a contradiction to  $e(XXY)\ge \gamma^2 n^3$. Thus, we have $|E_1|\geq3\gamma^2n^2$. Note that for $xy\in E_1$, we have $d_X(xy)\geq(1/2-\gamma-\gamma^2)n$ and hence $(1/2-\gamma-\gamma^2)n\leq|X|,|Y|\leq(1/2+\gamma+\gamma^2)n.$

Let $Y'=\{y\in Y: d_{E_0}(y)\leq\gamma^2n\}.$
%be the vertex set each $y\in Y'$  has degree $d_{E_0}(y)\leq\gamma^2n$.
Since $|E_0|\leq\gamma^4 n^2$, there are at most $\gamma^2n$ vertices $y$ in $Y$ such that $d_{E_0}(y)>\gamma^2n$. Thus we have $|Y'|\geq|Y|-\gamma^2 n$.

We claim that either {$d_{E_1}(y)\leq 3\gamma^2n$ or $d_{E_1}(y)\geq |X|-3\gamma n$,} for all $y\in Y'$.
Fix $y\in Y'$. Let $e_y$ be the number of edges $x_1x_2y$ of the form $XXY$ such that exactly one of $\{x_1y, x_2y\}$ belongs to $E_1$.
On one hand, we have
$$e_y\ge((1/2-\gamma-\gamma^2)n-d_{E_1}(y))\cdot d_{E_1}(y)\ge (|X|-2\gamma n-2\gamma^2 n-d_{E_1}(y))\cdot d_{E_1}(y),$$
since for each $x\in N_{E_1}(y)$, there are at least $(1/2-\gamma-\gamma^2)n-d_{E_1}(y)$ edges $xx'y$ of the form $XXY$ with $x'\in N_{E_0}(y)\cup N_{E_2}(y)$ and $|X|\leq(1/2+\gamma+\gamma^2)n$. On the other hand, we have
$$e_y\le |X|\cdot d_{E_0}(y)+\gamma^2n\cdot d_{E_2}(y)\le 2\gamma^2n|X|,$$
since $d_X(x'y)<\gamma n^2$ for $x'y\in E_2$, and the last inequality holds since $d_{E_0}(y)\leq\gamma^2n$ {and $d_{E_2}(y)\le |X|$.}
Therefore, we have
$$(|X|-2\gamma n-2\gamma^2 n-d_{E_1}(y))\cdot d_{E_1}(y)\le 2\gamma^2n|X|,$$
And hence we have either $d_{E_1}(y)\leq 3\gamma^2n$ or $d_{E_1}(y)\geq |X|-3\gamma n$, for all $y\in Y'$.

%$$((1/2-\gamma-\gamma^2)n-d_{E_1}(y))d_{E_1}(y)\le e_y\le (1/2+\gamma+\gamma^2)n d_{E_0}(y)+\gamma^2nd_{E_2}(y),$$
%{\color{red} $$((1/2-\gamma-\gamma^2)n-d_{E_1}(y))\cdot d_{E_1}(y)\le e_y\le |X|\cdot d_{E_0}(y)+\gamma^2n\cdot d_{E_2}(y)\le 2\gamma^2n|X|,$$
%the first inequality holds since, for each $x\in N_{E_1}(y)$, there are at least $(1/2-\gamma-\gamma^2)n-d_{E_1}(y)$ edges $xx'y$ of the form $XXY$ with $x'\in N_{E_0}(y)\cup N_{E_2}(y)$; the second inequality holds since $d_X(x'y)<\gamma n^2$ for $x'y\in E_2$, and the last inequality holds since $d_{E_0}(y)\leq\gamma^2n$ {and $d_{E_2}(y)\le |X|$.}  Therefore, we have either $d_{E_1}(y)\leq 2\gamma^2n$ or $d_{E_1}(y)\geq |X|-2\gamma^2 n$, for all $y\in Y'$. REWRITE! CHECK AGAIN! X0506}

Let $Y_0=\{ y : d_{E_1}(y)\geq |X|-3\gamma n, y\in Y'\}$. Clearly, $$|Y_0|\geq \frac{|E_1|-(|Y|-|Y'|)|X|-|Y|\cdot3\gamma^2n}{|X|}\geq\frac{3\gamma^2n^2-\gamma^2n|X|-|Y|\cdot3\gamma^2n}{|X|}\geq\gamma^2 n.$$
%since $|E_1|\geq3\gamma^2n^2$.

Now we claim that there are at least {$(1-14\gamma)\binom{|X|}{2}$} pairs $x_1x_2\in \binom{X}{2}$  such that $d_Y(x_1x_2)\geq\frac 1{10}\gamma^2 n$.
%By double-counting the number of edges of type $XXY_0$, we have
Clearly, {
\begin{eqnarray*}
e(XXY_0)&&=\frac 12\sum_{x\in X,y\in Y_0} d_X(xy)\\
&&\geq \frac {|Y_0|(1/2-\gamma-\gamma^2)n(|X|-3\gamma n)}2\\
&&\geq \frac {|Y_0||X|(1-5\gamma)|X|(1-7\gamma)}2\\
&&\geq (1-12\gamma)\binom{|X|}{2}|Y_0|.
\end{eqnarray*}}
{ On the other hand, if the number of pairs $x_1x_2\in {\binom{X}2}$ with $d_{Y_0}(x_1x_2)\ge \frac 1{10}\gamma^2n$ is less than $(1-14\gamma)\binom{|X|}2$, we have
\begin{eqnarray*}
e(XXY_0)&=&\sum_{x_1x_2\in\binom{X}{2}}d_{Y_0}(x_1x_2)\\
        &<& (1-14\gamma)\binom{|X|}2|Y_0|+14\gamma\binom{|X|}2\frac 1{10}\gamma^2n\\
       &\le& (1-12\gamma)\binom{|X|}2|Y_0|-2\gamma\binom{|X|}2|Y_0|+14\gamma\binom{|X|}2\frac 1{10}|Y_0|\\
       &=&(1-12\gamma)\binom{|X|}2|Y_0|-\frac35\gamma\binom{|X|}2|Y_0|\\
       &<& (1-12\gamma)\binom{|X|}2|Y_0|,
\end{eqnarray*}
a contradiction.}

Next, we claim that there are at least $(\frac12-11\gamma)\binom{|X|}{2}$ pairs $x_1x_2\in \binom{X}2$  such that $d_X(x_1x_2)\geq \gamma n$.
%Since $|E_1|\leq \frac {n^2}8$, together with $(1/2-3\gamma)n\leq|X|\leq (1/2+3\gamma)n$, we have
In fact,
\begin{eqnarray*}
\sum_{x_1x_2\in\binom{X}{2}}d_X(x_1x_2)&=&\sum_{x_1x_2\in\binom{X}{2}}d_H(x_1x_2)-\sum_{x_1x_2\in\binom{X}{2}}d_Y(x_1x_2)\\
&\geq&\left(\frac 12-\gamma\right)n\cdot\binom{|X|}{2}-\frac 12\sum_{x\in X,y\in Y}d_X(xy)\\
&>&\left(\frac 12-\gamma\right)n\cdot\binom{|X|}{2}-\frac 12\left(\gamma^4n^2\cdot|X|+\frac {n^2}8\cdot|X|+\frac {n^2}4\cdot\gamma^2n\right)\\
&\geq&\left(\frac 12-\gamma\right)n\cdot\binom{|X|}{2}-\frac n2\left(\gamma^4n\cdot|X|+\frac {n}8\cdot|X|+\frac {n}4\cdot\gamma^2n\right)\\
%&\geq&\left(\frac 12-\gamma\right)n\cdot\binom{|X|}{2}-\frac n2\left(\gamma\binom{|X|}{2}+(1/2+\gamma)\binom{|X|}{2}+\gamma\binom{|X|}{2}\right)\\
&\geq&\left(\frac14-3\gamma\right)n\binom{|X|}{2},
\end{eqnarray*}
the third inequality holds since $d_X(xy)\le |X|$ for any $xy\in E_0\cup E_1$, $d_X(xy)<\gamma^2n$ for $xy\in E_2$ and $|E_0|<\gamma^4n^2$, $|E_1|\leq \frac {n^2}8$ and $|E_2|\le \frac {n^2}4$; the last inequality holds since $(1/2-3\gamma)n\leq|X|\leq (1/2+3\gamma)n$.
Since $$\frac{(\frac14-3\gamma)n\binom{|X|}{2}-\gamma n\binom{|X|}{2}}{|X|}\geq\frac{\frac14-4\gamma}{\frac12+3\gamma}\binom{|X|}{2}\geq\left(\frac12-11\gamma\right)\binom{|X|}{2},$$ there are at least $(\frac12-11\gamma)\binom{|X|}{2}$ pairs $x_1x_2\in\binom{X}{2} $ such that $d_X(x_1x_2)\geq \gamma n$.

Therefore, there are at least $(1-14\gamma+\frac12-11\gamma-1)\binom{|X|}{2}\geq\frac {n^2}{100}$ pairs $x_1x_2\in \binom{X}{2}$ such that $d_X(x_1x_2)\geq \gamma n$ and $d_Y(x_1x_2)\geq \frac 1{10}\gamma^2 n$.

 As what we have done in the proof of Claim~\ref{large E0}, for $x_1x_2\in \binom{X}{2}$, let $k'_{s,t}(x_1x_2)$ be the number of 3-partite 3-graphs $K^3({\{x_1\},\{x_2\}, M})$ with $|M\cap X|=s$ and $|M\cap Y|=t$. Similarly, given $S\subset X$ and $T\subset Y$ with $|S|=s$ and $|T|=t$, let $G'(S,T)$ be the graph with vertex set $X$ and edge set consisting of all pairs $x_1x_2\in \binom{X}2$ such that $x_1x_2z\in H$ for all $z\in S\cup T$. Then, by double counting, we have
\begin{equation*}
\sum_{x_1x_2\in \binom{X}2}k'_{s,t}(x_1x_2)=\sum_{S\in\binom{X}{s},T\in\binom{Y}{t}} |G'(S,T)|.
\end{equation*}
Since there are at least $\frac {n^2}{100}$ pairs $x_1x_2\in \binom{X}2$ such that $d_X(x_1x_2)\geq\gamma n$ and $d_Y(x_1x_2)\geq \frac 1{10}\gamma^2 n$, we have
\begin{equation*}
LHS=\sum_{x_1x_2\in\binom{X}2}\binom{d_X(x_1x_2)}{s}\binom{d_Y(x_1x_2)}{t}\geq\binom{\gamma n}{s}\binom{\frac 1{10}\gamma^2 n}{t}\cdot\frac{n^2}{100}\geq\gamma_1'n^{m+2}.
\end{equation*}
On the other hand,
\begin{eqnarray*}
RHS&=&\left(\sum_{|G'(S,T)|<\gamma_1' n^2}+\sum_{|G'(S,T)|\geq\gamma_1' n^2}\right)|G'(S,T)|\\
&\leq&\sum_{|G'(S,T)|<\gamma_1 n^2}\gamma_1' n^2+\sum_{|G'(S,T)|\geq\gamma_1' n^2}\binom{|X|}{2}.
\end{eqnarray*}
So the number of pairs $(S,T)$ such that $|G'(S,T)|\geq \gamma_1' n^2$ is at least
\begin{equation*}
\left(\gamma_1' n^{m+2}-\gamma_1' n^2\binom{|X|}{s}\binom{|Y|}{t}\right)\big/\binom{|X|}{2}\geq \frac{1}{2}\gamma_1' n^{m+2}\big/(n^2/3)\geq \gamma_1' n^m.
\end{equation*}
%pairs of $(S,T)$ such that $|G_{s,t}'(S,T)|\geq \gamma_1' n^2$.
Fix such a pair $(S,T)$, by Lemma~\ref{supersaturation},  $G'(S,T)$ contains at leat $\mu_0' n^{2m}$ copies of $K^{2}(m)$, which give us at least $\mu_0' n^{2m}$ copies of $K^3(m)$ of type $(2m+s,t)$. Therefore, $H$ contains at least $\gamma_1' n^m\cdot\mu_0' n^{2m}\geq\mu_1' n^{3m}$ copies of $K^3(m)$ of type $(2m+s,t)$.\quad \rule{1mm}{2mm}

%Let $\mu=min\{\mu_1,\mu_1'\}$. By claim \ref{large E0} and lemma \ref{small E0}, we have $\mathbf{u}_X-\mathbf{u}_Y\in L_{P,F}^{\mu}(H)$. Thus by lemma \ref{lattice}, $H$ is $(K^3(m), i'_0,\eta')$-closed. Now we can apply lemma \ref{abp} with $t=3m, i=i_0,\eta=\eta'$ and obtain the desired absorbing set.

This completes the proof.
% of Lemma~\ref{absorption} follows from Claim~\ref{large E0} and \ref{small E0}.
\end{proof}

\section{Extremal case}
In this section, we prove Lemmas~\ref{extremalK_m} and \ref{extremalC_4}.
Let $G$ and $H$ be two $k$-graphs on the same vertex set $V$ and let $G\setminus H$ be the graph $(V, E(G)\setminus E(H))$. Suppose that $|V|=n$ and $0\leq \alpha\leq 1$, we say a vertex $v\in H$ {\it $\alpha$-good} with respect to $G$ if $d_{G\setminus H}(v)\leq \alpha n^{k-1}$, otherwise call it {\it $\alpha$-bad}. %{\color{red} Note that for a $\alpha$-good vertex $v$, there exist at most $\alpha_1 n$ vertices such that $d_{H\setminus G}(v)\geq\alpha_1 n$ for some $\alpha_1$(one can set $\alpha_1=\sqrt{\alpha}$).}
We call $H$ $\alpha$-good with respect to $G$ if all of vertices in $H$ are $\alpha$-good with respect to $G$. First we deal with a special case when $H$ is $\alpha$-good with respect to the extremal graph. We need a lemma from \cite{lattice-based} which follow with some extra work from a perfect packing theorem of Lu and Sz\'ekely~\cite{p-p in almost complete}. Given $V=A\cup B$, let $\mathcal{D}[A,B]$ be the $k$-graph on $V$ consisting of all edges of type $AB^{k-1}$.

\begin{lem}[Lemma 6.1 in \cite{lattice-based}] \label{good D}
Let $K$ be a complete $k$-partite $k$-graph of order $t$ with the first part of size $a_1$. Given $0<\rho\ll1/m$ and a sufficiently large integer $n$, suppose $H$ is a $k$-graph on $n\in t\mathbb{Z}$ vertices with a partition of $V(H)=X\cup Y$ such that $a_1|Y|=(t-a_1)|X|$. Furthermore, assume that $H$ is $\rho$-good with respect to $\mathcal{D}[X,Y]$. Then $H$ contains a $K$-factor.
\end{lem}

\begin{lem}\label{good}
Let $\alpha, \epsilon$ be any given constants with $0<\epsilon\ll \alpha$ and $m$ be an integer. Suppose that $H$ is a 3-graph with large enough order $n$ and $V(H)$ has a partition $A\cup B$ with $||A|-|B||<\epsilon n$ such that $H$ is $\alpha$-good with respect to $\mathcal{B}[A,B]$. Then $H$ contains a $K^3(m)$-tiling covering all but at most $2\epsilon n$ vertices. Furthermore, if $n\in12m\mathbb{Z}$ and $|A|=|B|$. Then $H$ contains a $K^3(m)$-factor.
\end{lem}
\begin{proof} Without loss of generality, assume $|A|\le |B|$. Let $|A|=6mn'+s$ and $|B|=6mn'+t$, where $0\le s<6m$ and $t=|B|-|A|+s< 2\epsilon n$.
Let $A_0$ and $B_0$ be the sets obtained from $A$ and $B$ by deleting $s$ and $t$ vertices from $A$ and $B$, respectively. Then $|A_0\cup B_0|=12mn'\in 12m\mathbb{N}$. Let $H_0=H[A_0\cup B_0]$ and $n_0:=V(H_0)$. Then $H_0$ must be $\alpha'$-good with respect to $\mathcal{B}[A_0, B_0]$ for some constant $\alpha'>0$.
Partition $A_0$ into three subsets $A_1,~A_2,~A_2'$ with $|A_1|=3mn', |A_2|=mn'$ and $|A_2'|=2mn'$. Let $H_1=H'[A_1\cup B_0]$ and $H_2=H'[A_2\cup A_2']$. Then we have $|V(H_1)|=\frac{3}{4} n_0$ and $|V(H_2)|=\frac{1}{4}n_0$. One can examine that $H_1$ is $\frac{16}{9}\alpha'$-good with respect to $\mathcal{D}[A_1,B_0]$ and $H_2$ is $16\alpha'$-good with respect to $\mathcal{D}[A_2,A_2']$. Set $K=K^3(m)$ and $t=3m$. Applying  Lemma~\ref{good D} to $H_1$ and $H_2$ with parameters $\frac{16}{9}\alpha'$ and $16\alpha'$, we obtain $K^3(m)$-factors $\mathcal{M}_1$ in $H_1$ and $\mathcal{M}_2$ in $H_2$, respectively. Therefore, $\mathcal{M}_1\cup\mathcal{M}_2$ is a desired $K^3(m)$-factor of $H_0$.

 If $n\in12m\mathbb{Z}$ and $|A|=|B|$ then $H_0=H$. Hence $\mathcal{M}_1\cup\mathcal{M}_2$ is a $K^3(m)$-factor of $H$.
\end{proof}

\noindent{\bf Remark: } Note that, in the above proof, the $K^3(m)$-factors $\mathcal{M}_1$ and $\mathcal{M}_2$ have the following property:\\
(1) Each member in $\mathcal{M}_1$ (resp. $\mathcal{M}_2$) has type $(m,2m)$ (resp. $(3m,0)$) with respect to the partition $A\cup B$, and \\
(2) both $|\mathcal{M}_1|(\sim\frac n4)$ and  $|\mathcal{M}_2|(\sim\frac n{12})$ are large enough.

The following classical result~\cite{KST} also will be used.
\begin{lem}[K\"ov\'ari-S\'os-Tur\'an, 1954]\label{EST}
For all $t\ge s\ge2$, the Tur\'an function of the complete bipartite graph $K^2(s,t)$ is   $$ex_2(n,K^2(s,t))\leq\frac12((t-1)^{1/s}n^{2-1/s}+(s+1)n).$$
\end{lem}

\subsection{Proofs of Lemmas~\ref{extremalK_m} and \ref{extremalC_4}}
%\noindent{\bf Proofs of Lemmas~\ref{extremalK_m} and \ref{extremalC_4}:}
Since $H$ is $\gamma$-extremal, there is a partition  $V=A\cup B$ such that {$|A|\le |B|\le\lceil n/2\rceil$}  and $H$ is $\gamma$-extremal with respect to $\mathcal{B}[A,B]$. Set $\gamma_1=\sqrt{\gamma}$. By the definition of $\gamma$-extremal, all but at most $\gamma_1n$ vertices in $V$ are $\gamma_1$-good with respect to $\mathcal{B}[A,B]$. Let $A_0$ and $B_0$ be the sets of $\gamma_1$-bad vertices in $A$ and  $B$, respectively. Then $|A_0\cup B_0|\leq\gamma_1n$. For a vertex $x\in A_0\cup B_0$, we call it {\it $B$-acceptable} if {$|E(H_x)\cap E(K^2(A, B))|\ge \frac{n^2}{40}$};
%there are at least $n^2/40$ pairs $(a, b)$ of vertices where $a\in A,~\in B$ and $ab\in H_x$.
otherwise we call it {\it $A$-acceptable}. Note that $|E(H_x)|\ge\delta_1(H)\geq(n-1)(\lfloor n/2\rfloor-1)/2$. If $x$ is $A$-acceptable then $|E(H_x)\cap \binom A2 |\ge \frac 34 \binom{|A|}{2}$ and $|E(H_x)\cap \binom B2|\ge \frac 34 \binom{|B|}{2}$.
%edges in $\binom{A}{2}$ and $3/4 \binom{|B|}{2}$ edges in $\binom{B}{2}$ are in  $H_x$.
Now move all $A$-acceptable vertices into $A$  and $B$-acceptable vertices into $B$, we get a new partition $V=A'\cup B'$ with the property that

  1) $n/2-\gamma_1 n\leq|A'|, |B'|\leq n/2+\gamma_1 n$ (since $|A_0\cup B_0|\leq\gamma_1n$);

  2) $H$ $\gamma_2$-contains $\mathcal{B}[A',B']$ for some constant $\gamma_2\gg\gamma_1$.

Moreover, we can partition $A'$ into $A_1, A_2$ so that:

 A1) Every vertex in $A_1$ is $\gamma_2$-good with respect to $\mathcal{B}[A',B']$;

 A2) $|A_2|\leq\gamma_1n$;

 A3) for every $x\in A_2$, $|E(H_x)\cap \binom{A'}{2}|\ge \frac 23 \binom{|A'|}{2}$ and $|E(H_x)\cap \binom{B'}{2}|\ge \frac 23 \binom{|B'|}{2}$.

 %then both of least $2/3 \binom{|A|}{2}$ edges in $\binom{|A|}{2}$ and $2/3 \binom{B}{2}$ edges in $\binom{B}{2}$ ate in  $H_x$.

Similarly, there is a partition $B_1,~B_2$ of $B'$ so that:

 B1) Every vertex in $B_1$ is $\gamma_2$-good with respect to $\mathcal{B}[A',B']$;

 B2) $|B_2|\leq\gamma_1n$;

 B3) for every $x\in B_2$, $|E(H_x)\cap E(K^2(A', B'))|\ge \frac{n^2}{50}$.
 %if $x\in B_2$, then there are at least $n^2/50$ pairs $ab$ of vertices where $a\in A,~b\in B$ and $ab\in H_x$.

 Our strategy is to find vertex-disjoint $K^3(m)$-tiling $\mathcal{K}_1,\mathcal{K}_2,\mathcal{K}_3,\mathcal{K}_4$ in $H$ so that the union of them is a
$K^3(m)$-factor of $H$, in which $\mathcal{K}_1$ is so-called 'parity breaking' copies dealing with the case $|B|\not\equiv0 \pmod{2m}$, $\mathcal{K}_2$ covers all vertices in $A_2\cup B_2$, and $\mathcal{K}_3$ is used to
%adjust the size of $B$???
guarantee the divisibility condition required by Lemma~\ref{good} after removing the vertices covered by $\mathcal{K}_1$ and $\mathcal{K}_2$.
%and $\mathcal{K}_3$.
%removed to satisfying the divisibility condition in Lemma~\ref{good}.
Furthermore, $\mathcal{K}_1,\mathcal{K}_2,\mathcal{K}_3$ are all small enough such that the graph obtained by deleting $\mathcal{K}_1,\mathcal{K}_2,\mathcal{K}_3$ is $\gamma_3$-good for some constant $\gamma_3$. Finally, we apply Lemma~\ref{good} to obtain $\mathcal{K}_4$. %Together, $\mathcal{K}_1\cup\mathcal{K}_2\cup\mathcal{K}_3\cup\mathcal{K}_4$ is a $K^3(m)$-factor of $H$.

In Claims \ref{d-b-breaker for K_m} and \ref{d-b-breaker for C_4}, we  show that such 'parity breaking' copies of $K^3(m)$ (resp. $K_{m,m}^3$) do exist.

\begin{claim}\label{d-b-breaker for K_m}
If $\delta_2(H)\ge n/2+m^{1/m}n^{1-1/m}$, then $H$ contains either $2m-1$ disjoint copies of $K^3(m)$ of type $(m+1,2m-1)$ or $2m-1$ disjoint copies of $K^3(m)$ of type $(3m-1, 1)$.
\end{claim}

\noindent{\it Proof of Claim \ref{d-b-breaker for K_m}:}
%Suppose that $\delta_2(H)\geq n/2+n^{1-1/m}$.
%We will show for any vertex set $W\subset V$ with $|W|\leq C$ for some constant $C\geq 6m^2$, there exists a copy of $K^3(m)$ of type $(m+1,2m-1)$ or $(3m-1,1)$ avoiding $W$.  %Choose such a vertex set $W$.
If we can find  a copy of $K^3(m)$ of type $(m+1,2m-1)$ or $(3m-1,1)$ avoiding any given vertex set $W\subset V$ with $|W|\leq C$ for some constant $C\geq 6m^2$, then {we can greedily find $2m-1$ disjoint copies of $K^3(m)$ of desired type because we always can find a new copy of $K^3(m)$ avoiding the vertices of copies of $K^3(m)$ we have found (since $C\geq 6m^2$)}. So the rest of the proof is to show the statement is true.  Choose any vertex set $W\subset V$ with $|W|\leq C$ for some constant $C\geq 6m^2$.
We split the proof into two cases depending on the size of $B'$.

First assume that $|B'|\leq n/2$. For any $a\in A', b\in B'$, we have $|N_H(ab)\cap A'|\geq m^{1/m}n^{1-1/m}$ since $\delta_2(H)\geq n/2+m^{1/m}n^{1-1/m}$. Construct an auxiliary bipartite graph $G$ as follows: set $V(G)=A'\cup B'$ and $E(G)$ consists of all pairs $ab$ with $a\in A', b\in B'$ and $|N_H(ab)\cap B'|\geq(1-\sqrt{\gamma_2})|B'|$. Since $H$ $\gamma_2$-contains $\mathcal{B}[A', B']$, there are at most $\gamma_2 n^3$ $A'B'B'$-edges missing in $H$. {Clearly, we have that at most $2\gamma_2 n^3/(\sqrt{\gamma_2}|B'|)\leq 8\sqrt{\gamma_2}n^2$ pairs $ab$ missing in $G$.} {By double-counting the number of ordered pairs $(v, e)$ with $v\in A'\setminus W$ and $e\in N_H(v)\cap E(G-W)$, we have
{$$\sum_{v\in A'\setminus W}|N_H(v)\cap E(G-W)|\geq {(|G|-Cn)}\cdot(m^{1/m}n^{1-1/m}-|A'\cap W|).$$
Note that $(|G|-Cn)(m^{1/m}n^{1-1/m}-|A'\cap W|)/|A'\setminus W|\geq \frac12(m-\frac12)^{1/m}{}n^{2-1/m}$. }We can choose a vertex $v\in A'\setminus W$ such that $|N_H(v)\cap E(G-W)|\geq \frac12(m-\frac12)^{1/m}{}n^{2-1/m}$. Lemma \ref{EST} implies that there exists a copy of $K^{2}(m)$, denoted by $M$, in $N_H(v)\cap E(G-W)$. By the definition of $E(G)$, $$\left|\left(\bigcap_{e\in M}N_H(e)\right)\cap\left(B'\setminus W\right)\right|\geq |B'|-m^2\sqrt{\gamma_2}|B'|-C\geq m-1$$
for sufficiently large $n$ and small $\gamma_2$.
Pick such any $m-1$ vertices together with $v$ and $V(M)$, we obtain a copy of $K^3(m)$ of type $(m+1,2m-1)$ avoiding $W$.}
 %We can greedily pick $2m-1$ disjoint copies of $K^3(m)$ of type $(m+1,2m-1)$, sice $C\geq 6m^2$.

Now assume $|B'|>n/2$. For any pair $aa'\in\binom{A'}{2}$, we have $|N_H(aa')\cap B'|\geq m^{1/m}n^{1-1/m}$.
%We will show $H$ contains $2m-1$ disjoint copies of $K^3(m)$ of type $(3m-1, 1)$.
Construct another auxiliary graph $G'$ as follows: set $V(G')=A'$ and $E(G')$ consists of all pairs $aa'\in \binom{A'}2$ with $|N_H(aa')\cap A'|\geq(1-\sqrt{\gamma_2})|A'|$. Similarly, {since there are at most $\gamma_2 n^3$ $A'A'A'$-edges missing in $H$, there are at most $3\gamma_2 n^3/(\sqrt{\gamma_2}|A'|)\leq 8\sqrt{\gamma_2}n^2$ edges $aa'$ missing in $G'$.} {By double-counting the number of ordered pairs $(v, e)$ with $v\in B'\setminus W$ and $e\in N_H(v)\cap E(G'-W)$, we have
{$$\sum_{v\in B'\setminus W}|N_H(v)\cap E(G'-W)|\geq (|G'|-C|A'|)\cdot(m^{1/m}n^{1-1/m}-|B'\cap W|).$$
Note that $(|G'|-C|A'|)(m^{1/m}n^{1-1/m}-|B'\cap W|)/|B'\setminus W|> \frac12m^{1/m}|A'|^{2-1/m}$.}
%{\color{red} WHY HERE WE DO NOT USE $(m-1/2)^{1/m}$. X0507 Actually I want to use $m^{1/m}$ but it does not hold in case "$|B|\leq n/2$")}
We can choose a vertex $v\in B'\setminus W$ such that $|N_H(v)\cap E(G'-W)|> \frac12m^{1/m}|A'|^{2-1/m}$.
%Thus, there exists a vertex $v\in B$ such that $|N(v)\cap (G'- W)|\geq (|G'|-Cn)(n^{1-1/m}-|B\cap W|)/|B\setminus W|\geq \frac{1}{2\sqrt{2}}n^{2-1/m}$.
Lemma \ref{EST} implies that there is a copy of $K^{2}(m)$, denoted by $M'$, in $N(v)\cap E(G'- W)$. By the definition of $E(G')$,
$$\left|\bigcap_{e\in M'}N(e)\cap (A'\setminus W)\right|\geq |A'|-m^2\sqrt{\gamma_2}|A'|-C\geq m-1.$$
Pick any such $m-1$ vertices together with $v$ and $V(M')$, we obtain a copy of $K^3(m)$ of of type $(3m-1,1)$ avoiding $W$.
%We can greedily pick $2m-1$ disjoint copies of $K^3(m)$ of type $(m+1,2m-1)$, sice $C\geq 6m^2$.
This completes the proof of claim \ref{d-b-breaker for K_m}. \quad \rule{1mm}{2mm}

\begin{claim}\label{d-b-breaker for C_4}
If $\delta_2(H)$ satisfies (\ref{codegree condition}) in Theorem~\ref{mainC4}, then $H$ contains a copy $K$ of $K^3_{m,m}$ of type $(m+1,2m-1)$ or $(3m-1,1)$, unless $|B'|=\lfloor n/2\rfloor$ when $n\equiv1\pmod4$. Furthermore, for any $0\leq t\leq m$, $H$ contains a copy $K'$ of $K^3_{m,m}$ of type $(m+2t,2m-2t)$ disjoint from $K$.
\end{claim}

\noindent{\it Proof:} If there exists a pair $a_1a_1'\in \binom{A'}2$ such that $|N_H(a_1a_1')\cap B'|\geq 2\gamma_1 n$, then we can choose $m$ distinct vertices $b_1,\ldots,b_m\in N_H(a_1a_1')\cap B_1$ since $2\gamma_1 n-|B_2|>m$. Note that for a $\gamma_2$-good vertex $b\in B'$,
$$|E(H_b)\cap E(K^2(A',B'))|\geq |A'||B'|-\gamma_2 n^2>\frac m{m+1}|A'||B'|,$$
we have $$\left|\bigcap_{i=1}^{m} E(H_{b_i})\cap E(K^2(A',B'))\right|\geq \frac 1{m+1}|A'||B'|.$$ Thus, $\bigcap_{i=1}^{m}E(H_{b_i})\cap E(K^2(A',B'))$ contains a matching of order $m-1$, choose such a matching $a_2b_2', \ldots, a_mb_m'$.
%{\color{red} distinct  $a_2,\ldots,a_m\in A'\setminus\{a_1,a_1'\},~b_2'\ldots,b_m'\in B'$ such that $b_1,\ldots,b_m\in N_H(a_ib_i')$, $i=2,\ldots,m$. ??? (WHY $a_ib_i'$s are matching in $\bigcap_{i=1}^{m} (E(H_{b_i})\cap E^2(K(A',B'))$ 0416X)}
So the subgraph of $H$ induced by $\{a_1', a_1, a_2,\ldots, a_m\}\cup \{ b_1, \ldots, b_m\}\cup\{ b_2', \ldots, b_m'\}$ contains a copy of $K^3_{m,m}$ of type $(m+1,2m-1)$.

Now assume $|N_H(a_1a_2)\cap B'|<2\gamma_1 n$ for any $a_1a_2\in \binom {A'}2$. Then $|N_H(a_1a_2)\cap A'|>n/2-2-2\gamma_1 n$. Let $F$ be the spanning subgraph consisting of all the edges of type $A'A'B'$ of $H$. We claim that if there is some $b\in B'$ such that $|F_b|>2m\gamma_1 n$, then $H$ contains a copy of $K^3_{m,m}$ of type $(3m-1,1)$. In fact, assume that there is some $b\in B'$ with $|F_b|>2m\gamma_1 n$.
First, suppose that $F_b$ contains a matching of size $m$. Let $a_1a_1',\ldots,a_ma_m'$ be a matching of $F_b$. Since $$\left|\bigcap_{i=1}^{m}N_H(a_ia_i')\cap A'\right|>m(n/2-2-2\gamma_1 n)-(m-1)|A'|>|A'|/2,$$
one can choose $m-1$ distinct vertices  $a_1'',\ldots,a_{m-1}''\in \bigcap_{i=1}^{m}N_H(a_ia_i')\cap A'$. And then the edges $a_ia_i'a_j'',a_ia_i'b\in E(H)$ $( i\in [m], j\in [m-1])$ form a copy of $K^3_{m,m}$ of type $(3m-1,1)$. Now suppose that $M$ is a maximum matching in $F_b$ of size at most $m-1$. Clearly, $V(M)$ is a vertex cover of $F_b$ and thus there exists a vertex $a$ in $V(M)$ of degree at least $\frac{2m\gamma_1n}{2(m-1)}\geq |A_2|+ m$. That is to say, there are $m$ distinct {$\gamma_2$}-good vertices $a_1'',\ldots,a_m''$ in {$N_{A'}(ab)$}. Note that for a $\gamma_2$-good vertex $a_i''\in A'$, $|E(H_{a_i''})\cap E(K^1(A'))|\geq \binom{|A'|}{2}-\gamma_2 n^2>\frac m{m+1}\binom{|A'|}{2}$, we have
$$\left|\bigcap_{i=1}^{m}E(H_{a_i''})\cap E(K^1(A'))\right|\geq \frac 1{m+1}\binom{|A'|}{2}.$$
Thus $\bigcap_{i=1}^{m}E(H_{a_i''})\cap E(K^1(A'))$ contains a matching of order $m-1$, choose such a matching $a_2a_2',\ldots, a_ma_m'$. Therefore, the subgraph of $H$ induced by $\{a_1'', \ldots, a_m''\}\cup \{a_2, a_2', \ldots, a_m, a_m'\}\cup\{a,b\}$ contains a copy of $K^3_{m,m}$ of type $(3m-1, 1)$, as desired. So the rest of the case is to show that such a vertex $b\in B'$ with $|F_b|\ge 2m\gamma_1n$ does exist.

If $n\equiv1\pmod{4}$ and $|B'|\leq\lfloor n/2\rfloor-1$ or $n\not\equiv 1\pmod{4}$ and $|B'|\leq\lceil n/2\rceil-1$, then for every pair $ab$ with $a\in A',~b\in B'$, we have $|N_H(ab)\cap A'|\geq 1$. Hence for any $b\in B'$, we have $\delta(F_b)\geq 1$ and so $|F_b|\geq|A'|/2\geq 2m\gamma_1n$, we are done. Now assume $|B'|\geq \lceil n/2\rceil$. Then for any pair $aa'\in\binom{A'}{2}$, we have $|N_H(aa')\cap B'|\geq1$. Since $$\binom{|A'|}{2}/|B'|\geq\binom{n/2-\gamma_1 n}{2}/(n/2+\gamma_1 n)>2m\gamma_1n,$$
 there exist at least one vertex $b\in B'$ such that $|F_b|>2m\gamma_1n$.

Next, we show that $H$ contains a copy $K'$ of $K^3_{m,m}$ of type $(m+2t,2m-2t)$ disjoint from $K$, $0\leq t\leq m$.
%{\color{blue} Choose any $m$ distinct $\gamma_2$-good vertices $a_1,\ldots,a_t\in A'$ and $b_{t+1},\ldots,b_m\in B'$. And there exist $a_1',\ldots,a_t'\in A'$ and $b_{t+1}',\ldots,b_m'\in B'$ such that $N_{A'}(a_ia_i')\geq|A'|-\sqrt{\gamma_2}n(1\leq i\leq t)$ and $N_{A'}(a_ia_i')\geq|A'|-\sqrt{\gamma_2}n,t+1\leq i\leq n$. Then $|\cap_{i=1}^t N_{A'}(a_ia_i')\cap_{i=t+1}^mN_A(b_ib_i)\cap A'|\geq |A'|/2$. Pick any $m$ vertices $a_1'',a_m''\in \cap_{i=1}^t N_{A'}(a_ia_i')\cap_{i=t+1}^mN_A(b_ib_i)\cap A'$. Clearly, these edges induce a copy of $K^3_{m,m}$ of type $(m+2t,2m-2t)$ in $H$ and all vertices can be pocked avoiding $K$.}
{Choose any $m$ distinct $\gamma_2$-good vertices $a_1,\ldots,a_t\in A'\setminus V(K)$ and $b_{t+1},\ldots,b_m\in B'\setminus V(K)$. Since $|E(H_{a_i})\cap\binom{A'}{2}|\geq \binom{|A'|}{2}-\gamma_2 n^2$, there exists at least $6m+1$ vertices $a'\in A'$ with $|N_{H_{a_i}}(a')\cap A'|\ge |A'|-\sqrt{\gamma_2}n$, that is we can choose $t$ distinct vertices $a_1',\ldots,a_t'\in A'\setminus V(K)$  such that $|N_{A'}(a_ia_i')|\geq|A'|-\sqrt{\gamma_2}n$ for $1\leq i\leq t$.
Similarly, since $|E(H_b)\cap E(K^2(A',B'))|\geq |A'||B'|-\gamma_2 n^2$, we can choose $m-t$ distinct vertices $b_{t+1}',\ldots, b_m'\in B'\setminus V(K)$ such that $|N_{A'}(b_jb_j')|\geq|A'|-\sqrt{\gamma_2}n$ for  $t+1\leq j\leq n$. Therefore, we have   $|\left(\cap_{i=1}^t N_{A'}(a_ia_i')\right)\cap\left(\cap_{i=t+1}^mN_A(b_ib_i)\right)\cap A'|\geq |A'|/2$. So we can pick $m$ vertices $a_1'', \ldots, a_m''\in \left(\cap_{i=1}^t N_{A'}(a_ia_i')\right)\cap\left(\cap_{i=t+1}^mN_A(b_ib_i)\right)\cap A'$ different from $a_i, b_i, a_i', b_i'$, $i\in[m]$. Clearly, the subgraph of $H$ induced by $\{a_i, a_i' : i\in [t]\}\cup\{b_i, b_i' : t+1\le i\le m\}\cup \{a_i'' : i\in [m]\}$ contains a copy of $K^3_{m,m}$ of type $(m+2t,2m-2t)$.} This completes the proof. \quad \rule{1mm}{2mm}

\medskip

The next claim shows that we can find a small $K^3(m)$-tiling to cover the vertices in $A_2\cup B_2$.

\begin{claim}\label{cover bad}
Suppose that $\delta_2(H)\ge \lfloor\frac n2\rfloor-1$. Let $W\subset V(H)$ with $|W|\leq\gamma_2 n$. Every vertex $x\in (A_2\cup B_2)\setminus W$ can be covered by a copy of $K^3(m)$ of type $(m,2m)$ avoiding $W$.
\end{claim}
\noindent{\it Proof of Claim \ref{cover bad}:} Recall that every vertex in $A_1\cup B_1$ is $\gamma_2$-good with respect to $\mathcal{B}[A,B]$. Let $G$ be the graph on vertex set $V$ and edge set consisting of all pairs $xy\in \binom{V}2$ satisfying $d_{\mathcal{B}\setminus H}(xy)\leq \sqrt{\gamma_2}n$. By the definition of $\gamma_2$-good, for each vertex $x\in A_1\cup B_1$, we have $d_G(x)\geq n-\sqrt{\gamma_2}n$.

If $x\in A_2\setminus W$, by A3), we have {$|H_x[B_1\setminus W]|\geq\frac23\binom{|B'|}{2}-\gamma_1n^2-2\gamma_n^2\geq\frac 12\binom{|B'|}{2}$.}
 Hence {$|E(H_x[B_1\setminus W])\cap E(G)|\geq \frac 13\binom{|B'|}{2}$.} Thus, by Lemma~\ref{EST}, $H_x[B_1\setminus W]\cap G$ contains a copy of $K^{2}(m)$, denoted by $M$.
 %And for any $e\in M$, $e\cup\{x\}\in H$.
Since $d_H(e)\ge |A'|-\sqrt{\gamma_2}n$ for any $e\in M$, {we have $|\bigcap_{e\in M}N_H(e)\cap A'|\ge |A'|-m^2\sqrt{\gamma_2}n$.  Hence we can choose $\{a_1,\ldots,a_{m-1}\}\subset\left(\bigcap_{e\in M}N_H(e)\cap A'\right)\setminus W$.
%such that $e\cup\{a_i\}\in H$ for any $e\in M,1\leq i\leq m-1$.
} Therefore, the subgraph of $H$ induced by $\{x,a_1,\ldots,a_{m-1}\}\cup V(M)$ contains a copy of $K^3(m)$ of type $(m,2m)$ covering $x$.

Now suppose $x\in B_2\setminus W$. B3) together with A2), B2) imply that {
$$|E(H_x)\cap E(G[A_1\setminus W,B_1\setminus W])|\geq \frac {n^2}{50}-2\gamma_1n^2-\gamma_2n^2-\sqrt\gamma_2n^2\geq\frac 1{100}|A'||B'|.$$
%$\sqrt\gamma_2n^2$ NEW ADDED ! X0508}
By Lemma \ref{EST}, $H_x\cap G[A_1\setminus W,B_1\setminus W]$ contains a copy of $K^{2}(m)$ avoiding $W$, denoted by $M'$.} Since $d_H(e)\geq |B'|-\sqrt{\gamma_2}n$ for any $e\in M'$, we have $|\bigcap_{e\in M}N_H(e)\cap B'|\ge |B'|-m^2\sqrt{\gamma_2}n$. Hence we can choose $m-1$ distinct vertices $b_1,\ldots,b_{m-1}\in \left(\bigcap_{e\in M'}N_H(e)\cap B'\right)\setminus W$.
%such that $e\cup\{b_i\}\in H$ for any $e\in M',1\leq i\leq m-1$.
Therefore, the subgraph of $H$ induced by $\{x,b_1,\ldots,b_{m-1}\}\cup V(M)$ contains a copy of $K^3(m)$ of type $(m,2m)$ covering $x$, as desired. \quad \rule{1mm}{2mm}

\medskip

\noindent{\it Proof of Lemma~\ref{extremalK_m}:} Let $t\equiv|B'| \pmod{2m}$ such that $0\leq t\leq 2m-1$. Let $\mathcal{K}_1$ be $2m-t$ disjoint copies of $K^3(m)$ of type $(m+1,2m-1)$ or $t$ disjoint copies of $K^3(m)$ of type $(3m-1, 1)$ in $H$ guaranteed by Claim~\ref{d-b-breaker for K_m}. Note that $|V(\mathcal{K}_1)|\le 6m^2$ is small enough. We can apply Claim \ref{cover bad} recursively to $H$ to obtain a $K^3(m)$-tiling $\mathcal{K}_2$ covering  all vertices of $(A_2\cup B_2)\setminus V(\mathcal{K}_1)$.
% and disjoint from $\mathcal{K}_1$ in $H$.
%so that $\mathcal{K}_2$ covers all vertices in $(A_2\cup B_2)\setminus V(\mathcal{K}_1)$ and $|\mathcal{K}_2|\leq 2\gamma_1n$.
Moreover, every copy of $K^3(m)$ in $\mathcal{K}_2$ is of type $(m, 2m)$. Let $A'':= A'\setminus V(\mathcal{K}_1\cup\mathcal{K}_2)$ and $B'':= B'\setminus V(\mathcal{K}_1\cup\mathcal{K}_2)$. Clearly, $|B''|\equiv 0\pmod {2m}$. Since $n\in 3m\mathbb{N}$, we have $|A''\cup B''|\equiv0\pmod{3m}$ and $|A''|\equiv0\pmod{m}$. Since $|\mathcal{K}_1|< 2m$ and $|\mathcal{K}_2|\leq2\gamma_1n$, we have $n/2-5m\gamma_1n\leq|A''|, |B''|\leq n/2+\gamma_1n$. Let $|A''|=(6a+s)m, |B''|=(6b'+2t')m$. Then it is easy to check that $s\equiv t'\pmod3$. So we can set $|A''|=(6a+s)m$ and $|B''|=(6b+2s)m$ for some $0\leq s\leq5$. {Now, each vertex in $A''\cup B''$ is $\gamma_2'$-good {with respect to $\mathcal{B}(A'', B'')$} for some constant $\gamma_2'\gg\gamma_2$.}
{By Lemma~\ref{good}, we can find  $6(b-a)+s$ disjoint copies of $K^3(m)$ of type $(m,2m)$ if $b-a\geq0$, or $2(a-b)$ disjoint copies of $K^3(m)$ of type $(3m,0)$ and $s$ disjoint copies of $K^3(m)$ of type $(m,2m)$ if $b-a<0$. Let $\mathcal{K}_3$ be these copies of $K^3(m)$.}
Thus, {$|\mathcal{K}_3|\leq 6|b-a|+s\le 6\gamma_1 n$.}
Let $A^*=A''\setminus V(\mathcal{K}_3)$ and $B^*=B''\setminus V(\mathcal{K}_3)$. Then we have {$|A^*|=|B^*|\equiv 0\pmod{6m}$
%(in the former case $|A^*|=|B^*|=12a-6b$ and the latter case $|A^*|=|B^*|=6b$)}
and $|A^*|=|B^*|\geq n/2-100\gamma_1 mn$. Clearly, $A^*\subset A_1$ and $B^*\subset B_1$. Let $H^*=H[A^*\cup B^*]$. Since both $|A_2|$ and $|B_2|$ are small, it can be checked that there is some constant $\gamma_3\gg\gamma_2'$ such that every vertex in $H^*$ is $\gamma_3$-good with respect to $\mathcal{B}[A^*,B^*]$. By Lemma \ref{good},  $H^*$ contains a $K^3(m)$-factor, say $\mathcal{K}_4$. Therefore,  $\mathcal{K}_1\cup\mathcal{K}_2\cup\mathcal{K}_3\cup\mathcal{K}_4$ is a $K^3(m)$-factor of $H$. \quad \rule{1mm}{2mm}

\medskip

\noindent{\it Proof of Lemma~\ref{extremalC_4}:} The proof is similar to the one of Lemma \ref{extremalK_m}. Note that $n\in 3m\mathbb{N}$ and $\delta_2(H)$ satisfies condition (\ref{codegree condition}). Let $t\equiv|B'|\pmod {2m}$ with $0\leq t\leq 2m-1$. If $t$ is even (note that $|B'|=\lfloor n/2\rfloor$ and $n\equiv 1\pmod 4$ belongs to this case), by Claim \ref{d-b-breaker for C_4} for $m-t/2$, we can find a copy $K'$ of $K^3_{m,m}$ of type $(3m-t, t)$ in $H$.  Set $\mathcal{K}_1=\{K'\}$. Now assume $t$ is odd.
%By Claim~\ref{d-b-breaker for C_4},
Then we can find two disjoint copies $K, K'$ of $K_{m,m}^3$ of types $(m+1,2m-1)$ and $(3m-t-1, t+1)$ (by Claim~\ref{d-b-breaker for C_4} for $m-(t+1)/2$)   or of types $(3m-1,1)$ and $(3m-t+1, t-1)$ (by Claim~\ref{d-b-breaker for C_4} for $m-(t-1)/2$).     In this case, set $\mathcal{K}_1=\{K,K'\}$.
%$K'$ is of type $(3m-t',t')$, where $t'=|B\setminus K|$ mod $2m$.
For each case, we have  $|B'\setminus V(\mathcal{K}_1)|\equiv 0\pmod{2m}$ and $|A'\setminus V(\mathcal{K}_1)|\equiv 0\pmod{m}$.
%$\mathcal{K}_2,\mathcal{K}_3,\mathcal{K}_4$ do exist following from the $K^3(m)$-tillings.
Since $K^3_{m,m}$ is a spanning subgraph of $K^3(m)$, the existence of $K^3_{m,m}$-tiling $\mathcal{K}_2,\mathcal{K}_3,\mathcal{K}_4$ follows from that of $K^3(m)$-tillings $\mathcal{K}_2,\mathcal{K}_3,\mathcal{K}_4$ in $H$ with the same argument as in Lemma~\ref{extremalC_4}.
%We only show the existence of $\mathcal{K}_1$.
Finally we have $\mathcal{K}_1\cup\mathcal{K}_2\cup\mathcal{K}_3\cup\mathcal{K}_4$ is a $K^3_{m,m}$-factor of $H$. \quad \rule{1mm}{2mm}

\vspace{5pt}

\noindent {\bf Acknowledgement. } The research is
supported by National Natural Science Foundation of China (no.  11671376) and National Natural Science Foundation of Anhui Province (no. 1708085MA18).

\end{document}